\documentclass{amsart}
\usepackage{amssymb, epsfig}
\usepackage{amssymb, amsfonts, amsmath, amsthm, marvosym}
\usepackage{graphicx}
\usepackage{enumerate}
\usepackage{url} 
\usepackage{float}
\usepackage{hyperref}

\newcommand{\R}{\mathbf{R}}
\newcommand{\E}{\mathbf{E}}

\newcommand{\Recess}{{\mathcal R}}
\newcommand{\Future}{{\mathcal F}}
\newcommand{\Past}{{\mathcal P}}
\newcommand{\Normals}{{\mathcal N}}

\newcommand{\II}{\mathrm{I}\hspace{-0.8pt}\mathrm{I}}

\newcommand{\M}{\E^{n+1}_1}
\newcommand{\SEu}{\E^{n+k}_k}
\newcommand{\m}{\E^n_1}

\newcommand{\Mink}{\E^3_1}
\newcommand{\mink}{\E^2_1}

\newcommand{\Eu}{\E^{n+1}}

\newcommand{\x}{\mathbf{x}}

\newcommand{\w}{\mathbf{w}}

\theoremstyle{plain}
\newtheorem{thm}{Theorem}[section]
\newtheorem{cor}[thm]{Corollary}
\newtheorem{prop}[thm]{Proposition}
\newtheorem{lem}[thm]{Lemma}
\newtheorem{definition}[thm]{Definition}
\newtheorem{claim}{Claim}

\theoremstyle{definition}

\newtheorem{remark}[thm]{Remark}

\newtheorem{example}[thm]{Example}

\DeclareMathOperator{\constant}{constant}

\DeclareMathOperator{\cl}{cl}

\DeclareMathOperator{\hess}{\nabla^2\!}
\DeclareMathOperator{\del}{\nabla\!}

\DeclareMathOperator{\Sec}{Sec}
\DeclareMathOperator{\inte}{int}

\renewcommand{\tilde}{\widetilde}

\DeclareMathOperator{\range}{range}

\numberwithin{equation}{section}
\begin{document}

\title{Convex functions and geodesic connectedness of space-times}

\author{Stephanie B. Alexander}
\address{1409 W. Green St., Urbana, Illinois 61801}
\email{sba@illinois.edu}
\thanks{This work was partially supported by a grant from the Simons Foundation (\#209053 to Stephanie Alexander).}
\author{William A. Karr}
\address{1409 W. Green St., Urbana, Illinois 61801}
\email{wkarr2@illinois.edu}
\thanks{This material is partially based upon work supported by the National Science Foundation Graduate Research Fellowship to William Karr under Grant No. DGE 11-44245. Any opinion, findings, and conclusions or recommendations expressed in this material are those of the authors and do not necessarily reflect the views of the National Science Foundation.}

\maketitle 

\begin{abstract}
 This paper explores the relation between convex functions and the geometry of space-times and semi-Riemannian manifolds. Specifically, we study geodesic connectedness. We give geometric-topological proofs of geodesic connectedness for classes of space-times to which known methods do not apply. For instance: A null-disprisoning space-time is geodesically connected if it supports a proper, nonnegative strictly convex function whose critical set is a point. Timelike strictly convex hypersurfaces of Minkowski space are geodesically connected. We also give a criterion for the existence of a convex function on a semi-Riemannian manifold. We compare our work with previously known results.
\end{abstract}

\tableofcontents

\section{Introduction}\label{introduction}

This paper explores the relation between geometric convexity, and geodesic connectedness of 
space-times and semi-Riemannian manifolds. We consider geodesics of all causal types, since they form the scaffolding for the global topological and geometric structure of the space. 

According to Gibbons and Ishibashi \cite{gi}: ``Convexity and convex functions play an important role in theoretical physics. For example, Gibbs's approach to thermodynamics \cite{gibbs} is based on the idea that the free energy should be a convex function. A closely related concept is that of a convex cone which also has numerous applications to physics. 
Perhaps the most familiar example is the light cone of Minkowski space-time. Equally important is the convex cone of mixed states of density matrices in quantum mechanics. 
Convexity and convex functions also have important applications to geometry, including Riemannian geometry \cite{udriste}. It is surprising therefore that, to our knowledge, 
that [sic] techniques making use of convexity and convex functions have played no great role in General Relativity.'' 

Sufficient conditions for geodesic 
connectedness of Lorentzian manifolds are given by an early theorem of Uhlenbeck \cite[Theorem 5.3]{uhlenbeck}, and by \cite[Theorem 11.25]{bee}. However, these theorems
concern spaces with no conjugate points, whereas the spaces we consider may have conjugate points along geodesics of all causal types.

Geodesic connectedness was studied via an infinite-dimensional variational theory introduced by Benci, Fortunato, Giannoni and Masiello at the end of the 1980s. 
For Lorentzian manifolds carrying a timelike 
or null 
Killing field, geodesic connectedness has only recently become well understood \cite{cfs,bcf}. 
It is also known to hold for globally hyperbolic space-times carrying time-dependent orthogonal splittings satisfying certain conditions, 
as summarized in Theorem \ref{thm:cs-main} of 
in the appendix. See the 
%S
%informative 
survey of geodesics in semi-Riemannian manifolds by Candela and Sanchez \cite{cs:survey}, 
%S
and  the book \cite{masiello:book} and review article \cite{masiello:survey} by Masiello. 

Globally hyperbolic manifolds always have orthogonal splittings \cite{bs}, but there may be none satisfying the conditions just mentioned, e.g. de Sitter space, which is not geodesically connected. Or there might exist splittings that satisfy the conditions, but no known way to determine their existence. 

According to \cite{cs:survey}, 
``it should be interesting to obtain a result similar to that one also under weaker assumptions on the metric or under intrinsic hypotheses more related to the geometry of the manifold.'' 

Uhlenbeck considers orthogonal splittings satisfying a metric growth condition, and also calls for a more geometric approach, observing that the growth condition is ``not very satisfactory since it depends on the splitting [which] may be changed in drastically different ways ... it is to be hoped that a similar condition that does not depend on coordinates may be found'' \cite[p.\,75]{uhlenbeck}.

Using convex functions, we give geometric/topological proofs of geodesic connectedness for classes of space-times to which known methods do not apply. Our theorems concern space-times that are strongly causal or, more generally, null-disprisoning 
(see Definition \ref{def:null-dispris}); 
or else timelike convex hypersurfaces  
%
%which while
%globally hyperbolic, 
%typically do not have natural orthogonal splittings that satisfy 
%all the conditions in \cite{cs:survey} 
%(see Appendix \ref{app:conclusion}), nor do we know how to determine if any splitting that satisfies the conditions exists.
(see Appendix \ref{app:conclusion}).

We remark that convexity properties of timelike hypersurfaces were used by Chru\'{s}ciel and Galloway for geometric arguments concerning the mass of asymptotically Schwarzschildian spacetimes \cite{ch-g}.
%S
Masiello  has studied the relation of geodesic connectedness to convex domains in Lorentzian manifolds (see Appendix \ref{app:conclusion}). In \cite{gmp}, Giannoni, Masiello and Piccione used convex functions on Riemannian manifolds to study the number of light rays in the framework of the gravitational lensing effect.

Convex hypersurfaces of $\Eu$ are Riemannian 
manifolds of sectional curvature $\Sec\ge 0$, and their properties reflect those of general Riemannian manifolds of $\Sec\ge 0$. Timelike convex hypersurfaces of $\M$ satisfy $\mathcal{R}\ge 0$. This condition, introduced and applied by Andersson and Howard \cite{ah}, extends $\Sec\ge 0$ from the Riemannian to the semi-Riemannian setting by requiring spacelike sectional curvatures to be $\ge 0$ and timelike ones to be $\le 0$ (similarly for $\mathcal{R}\ge \kappa\,$ and $\mathcal{R}\le \kappa$). Thus our motivation for studying timelike convex hypersurfaces is two-fold: They are space-times to which topological/geometric arguments readily apply, and in particular they carry convex functions. And as in the Riemannian case, they should 
be a guide to
properties of more general space-times of $\mathcal{R}\ge 0\,$ (for some properties of $\mathcal{R}\ge 0\,$, see Remark \ref{rem:R>K} below).

\section{Results}\label{results}

We take \cite{bee}, \cite{oneill}, \cite{he} as standard references on Lorentzian and semi-Riemannian geometry.

\begin{definition}
By a \emph{convex (strictly convex) function} on a 
semi-Riemannian manifold, we mean a smooth real-valued function whose restriction to every geodesic has nonnegative (positive) second derivative. Equivalently, $f$ is convex (strictly convex) if and only if the Hessian \,$\hess f$ is positive semidefinite (positive definite). 
\end{definition}

\begin{remark}
This paper demonstrates the importance of these classically convex functions (equivalently, taking the negative, concave functions) in studying 
certain space-times that satisfy 
the curvature condition $\mathcal{R}\ge 0\,$. On Riemannian spaces with 
sectional curvature $\ge 0$, such functions arise naturally (Cheeger-Gromoll\,\cite{c-g},  
also
see \cite {petersen}).

In \cite{gi}, Gibbons and Ishibashi introduce and consider ``space-time convex" functions on Lorentzian manifolds, namely those satisfying 
\begin{equation}\label{eq:space-time-convex}
\hess f(\mathbf x,\mathbf x) \ge c\,g(\mathbf x,\mathbf x), 
\end{equation} for any tangent vector $\mathbf x$, 
where $c>0$, $g$ is the Lorentzian metric, and $\hess f$ has Lorentzian signature.  
They discuss consequences of the existence of such functions, for instance, ruling out closed marginally inner and outer trapped surfaces. 
They give examples of space-time convex functions on cosmological space-times, anti-de-Sitter space and black-hole space-times, and consider level sets of convex functions, as well as foliations by constant mean curvature hypersurfaces.

On Riemannian spaces with sectional curvature $\le 0$, convex functions again arise naturally (Bishop-O'Neill\,\cite{b-o}, see examples in \cite[Ch.4]{udriste}).
In future work, we shall explore a close relationship between upper bounds on $\mathcal {R}$ and space-time convex functions.
\end{remark}

In an early and influential consideration of geodesic connectedness of Riemannian manifolds, Gordon proved that if a connected Riemannian manifold $M$ supports a proper, nonnegative convex function, then $M$ is geodesically connected \cite{gordon}. Gordon's proof depends on the fact that complete Riemannian manifolds are geodesically connected, and does not extend to the Lorentz setting where geodesic connectedness is not a consequence of any completeness hypothesis.

We prove the following semi-Riemannian version of Gordon's theorem.

\begin{definition}\label{def:null-dispris}
A semi-Riemannian manifold $M$ is called \emph{disprisoning} if for every inextendible geodesic $\gamma:(a,b)\to M$, neither end lies in a compact set. $M$ is called \emph{null-disprisoning} if for every inextendible null geodesic, neither end lies in a compact set.
\end{definition}

Note that strongly causal, in particular globally hyperbolic, space-times are null-disprisoning \cite[Proposition 3.13]{bee}.

\begin{thm}\label{thm:gordon-semi}
Let $M$ be a null-disprisoning semi-Riemannian manifold. Suppose $M$ supports a proper, nonnegative convex function $f:M\to\R$ whose critical set is a 
minimum point. If there is no non-constant complete geodesic on which $f$ is constant (for example, if $f$ is strictly convex), then $M$ is geodesically connected.
\end{thm}
 
\begin{remark}\label{rem:R>K} 
In Riemannian comparison theory, the existence of proper nonnegative convex functions plays a fundamental role. A complete Riemannian manifold of nonnegative sectional curvature always carries such a function, obtained by taking the negative of the infimum of all Busemann functions of rays based at a point. The Soul Theorem of Meyer-Cheeger-Gromoll is a consequence (see \cite[\S 11.4]{petersen}). 

We have already mentioned that timelike convex hypersurfaces $M$ of $\M$ satisfy $\mathcal{R}\ge 0$ (namely, timelike sectional curvatures $\le 0$ and spacelike ones $\ge 0$); see Proposition \ref{prop:hyp-R>0}. Moreover, they support proper convex functions (Theorem \ref{thm:support-convex}). We expect timelike convex hypersurfaces $M$ to indicate properties of more general space-times of $\mathcal{R}\ge 0$. Thus we come to the question: do space-times with $\mathcal{R}$ bounds have a rich structure analogous to Riemannian comparison theory? We mention some affirmative indicators:
\begin{enumerate}
\item
Andersson and Howard proved ``gap'' rigidity theorems for $\mathcal{R}\ge 0$ and $\mathcal{R}\le 0$ of the type first proved for Riemannian manifolds of $\Sec\ge 0$ and $\Sec\le 0$ by Greene and Wu \cite{gw} and Gromov \cite{bgs} respectively \cite{ah}.
\item
Using the Penrose trapped surface theorem, Mukuno recently proved an analogue of Myers' Theorem  for null-geodesically complete Lorentzian manifolds $M$ with metric of the form $-dt^2+g_{\text{\,Riem}}(t) $ and compact second factor. Namely, if $M$ satisfies $\mathcal{R}\ge \kappa$ for $\kappa>0$, then $M$ has finite fundamental group \cite{mukuno}.  
(See also \cite{mukuno1}.)
\item
In Riemannian manifolds, sectional curvature bounds are characterized by local distance comparisons. In \cite{ab-lorentz}, an analogous theorem is shown to hold in semi-Riemannian manifolds having an $\mathcal{R}$ bound.
\end{enumerate}
\end{remark}

\begin{definition}
The semi-Euclidean space $\SEu$ is the real vector space of dimension $n+k$ carrying a nondegenerate inner product whose diagonal form consists of $n$ positives followed by $k$ negatives.
\end{definition}
 
\begin{definition}\label{def:convex-hyp}
A \emph{convex body} in $\SEu$ is a 
closed 
convex set (not assumed compact) with nonempty interior. A \emph{convex hypersurface} of $\SEu$ is a connected smooth manifold that is smoothly embedded as the boundary of a convex body. 
\end{definition} 

We prove the following geodesic connectedness theorem for a timelike convex hypersurface $M$ of Minkowski space $\M$. By a 
\emph{rolled 
Euclidean half-plane in $M$}, we mean the image of 
an isometrically immersed Euclidean half-plane $\{(x,y):x\in\R,y\ge 0\}$,
where the images of 
half-lines $x=\constant$ are parallel half-lines in $\M$ and the image of the boundary $y=0$ lies in a compact set.

\begin{thm}\label{thm:main}
Let $M$ be a timelike convex hypersurface of $\M$. 
Suppose that after splitting off a semi-Euclidean factor of maximal dimension, 
$M$ contains no 
rolled
Euclidean half-plane.
Then $M$ is geodesically connected.

In particular, any timelike strictly convex hypersurface is geodesically connected. 
\end{thm}

\begin{remark}
The no rolled Euclidean half-plane condition is  technical and we would like to eliminate it. 
We use it here to prove that $M$ is disprisoning.
 \end{remark}

\begin{example}\label{ex:race-track}
An example of a timelike strictly convex surface is examined and illustrated in the appendix. 
Now we give an example of a timelike convex hypersurface $M$ of $\M$ 
such that $M$ contains a 
rolled
Euclidean half-plane if $n>2$. 

Let $C^{n-1}$ be a smoothly capped cylinder embedded as a convex hypersurface of the copy of Euclidean space $ \{x^{n+1}=0 \}$ in $\M$, $n\ge 2$. Let the cylindrical part of $C^{n-1}$ be given by
 \begin{equation}\label{eq:cylinder}
 x^1\ge 1,\ \ 
 |x^2|^2 + \cdots + |x^n|^2 = 1, \ \ x^{n+1}=0,
 \end{equation}
and let the cap lie in the region $0\le x^1\le 1$. Set
\[
M = \{ p + (\sqrt{1+t^2}, 0 ,\dots, 0, t): t\in\R, p\in C^{n-1}\}. 
\]
The claimed properties of $M$ are easily verified.
  \end{example}

The following corollary of Theorem \ref{thm:gordon-semi} generates yet more geodesically connected space-times:

\begin{definition} 
A smooth function $f:M\to\R$ on a Lorentzian manifold $M$ will be called \emph{Lorentzian} if the graph of $f$ in $M\times\R$ is a timelike submanifold, i.e. $\langle\del f,\del f \rangle> -1$.
\end{definition}

\begin{cor}\label{cor:lorentz-graph}
Let $M$ be a connected, strongly causal space-time, and $f:M\to\R$ be a proper, nonnegative strictly convex Lorentzian function.
Then the graph of $f$ in $M\times\R$ is geodesically connected.
\end{cor}

Geodesic connectedness can also be diagnosed using not-necessarily-convex functions. Specifically, 
in Theorem \ref{thm:convex-function1} we give a criterion for the levels of a function on a semi-Riemannian manifold to be the levels of a convex function. 

The criterion is related to Fenchel's criterion for deciding if a function defined on an affine space and having convex level sets can be reparametrized as a convex function \cite{fenchel}. 
Theorems \ref{thm:convex-function1} and \ref{thm:gordon-semi} allow us to extend our class of geodesically connected spaces.

Here is a special case of these theorems.
The negativity of the expression \,$\mu$\, measures how badly the function \,$u$\, fails to be convex. 

\begin{thm}\label{thm:strict-criterion}
Suppose $u :M\to\R$ is a proper 
smooth nonnegative function on a connected semi-Riemannian manifold $M$, where the critical set of $u$ is a minimum point $p_0$, say $u(p _0)=0$. 
For $a\in\range u - \{0\}$, suppose the level sets 
$L_a$ are infinitesimally strictly convex, i.e. $\hess u\,(\mathbf x, \mathbf x) > 0$ if $\mathbf x\in T_pL_a$.

Let $N$ be a vector field on $M-\{p_0\}$ satisfying $Nu>0$. 
For $a\in\range u - \{0\}$, set
\begin{align}
\notag
\mu\,(a) =\inf\,
\bigl\{\bigl[\hess u(\mathbf x, \mathbf x)\hess u(N_p,N_p)- \bigl(\hess u(\mathbf x,N_p)\bigr)^2\bigr]/ &(N_pu)^2\hess u( \mathbf x, \mathbf x):\\&
p \in L_a, \mathbf x\in T_pL_a
\notag
\bigr\}.
\end{align}

If $\mu$ is bounded below by a continuous function 
\,$h:\range u - \{0\}\to\R$ that extends continuously to $0$, 
then: 
\begin{enumerate}
\item\label{convex-existence}
There is a smooth function $\rho:[0,\infty)\to [0,\infty)$ such that \,$\rho' \geq 1$\, and \,$f=\rho \circ u$\, is a proper strictly convex function.
\item
If $M$ is null-disprisoning, then
$M$ is geodesically connected.
\item If $M$ is a strongly causal space-time and $u$ is Lorentzian, then the graph in $M\times\R$ of $u$ is geodesically connected.
\end{enumerate}
\end{thm}

\begin{remark}
In applications, it is often possible to verify the hypothesis on $\mu$ by showing that \,$\mu$\, is continuous and finite. 
\end{remark}

As an application, we construct a large class of non-convex Lorentzian hypersurfaces 
in $\M$ 
that are geodesically connected (Corollary \ref{warbled-examples}).

In the appendix, we give an example that is geodesically connected by Theorem \ref{thm:main}, yet does not appear to carry orthogonal splittings that satisfy the growth conditions required by
the theory discussed in \cite{cs:survey}. 

%%%%%%%%%%%%%%%%%%%%%%%%%%%%%%%%%%%%%

%%%%%%%%%%%%%%%%%%
%%%%%%%%%%%%%%%%%%
 
\section{Convex functions and geodesic connectedness}

\begin{definition}
Let $f:M\to\R$ be a smooth function on a semi-Riemannian manifold $M$. The \emph{Hessian of $f$} is the symmetric $(0,2)$ tensor field $\hess f$ defined by 
\[\hess f \,(X,Y) \,= \,X\,Y\, f - (\del_{\,X}\!Y)\,f.
\] 
\end{definition}

\begin{definition}
For a semi-Riemannian manifold $M$, $SM$ will denote the unit tangent bundle for some Riemannian metric $g_{\emph{\,Riem}} $ on $M$. 
When we write $SM$ or $S_p\,M$, it means we have made a choice of $g_{\emph{\,Riem}} $. 
\end{definition}

Throughout this section, for a given convex function $f:M \to \R$ 
and any non-critical value $a$ of $f$, we denote the level sets by  $L_a=\{p\in M: f( p)=a\}$ and the sublevel sets by $M_a=\{p\in M: f( p)\le a\}$. 

\begin{lem}\label{lem:exit-or-level}
Let $M$ be a null-disprisoning semi-Riemannian manifold, and $f:M \to \R$ be a nonnegative proper convex function. Suppose the critical set $C$ of $f$ is connected, so $C$ is the minimum set, say $f|C=0$. 
Then one of these two statements holds:

\begin{enumerate}
\item\label{geos-exit}
$M$ is disprisoning,
\item There is a complete non-constant geodesic $\gamma$ such that $f\circ\gamma $ is constant.\end{enumerate}
\end{lem}

\begin{proof}
Since $M$ is null-disprisoning, $M$ is noncompact. Since $f$ is proper, the values of $f$ are unbounded.
 
Suppose $M$ is not disprisoning. Then there exists $p\in M_a$, and a right-sidedly maximally extended geodesic $\alpha$ with left-hand endpoint $\alpha(0) = p$, such that $\alpha$ does not leave $ M_a$. 	

Suppose $\alpha$ is defined on $[0,b)$, $b\le \infty$. Consider an increasing sequence $t_i\to b$, and a sequence $(\alpha(t_i), \mathbf v_i)\in SM$ where $\mathbf v_i$ has the same direction as $\alpha'(t_i)$. Since the $\alpha(t_i)$ lie in a compact set, we may assume $(\alpha(t_i),\mathbf v_i)\to (q,\mathbf v)\in SM$.

\begin{claim}\label{cl:non-null}
$\mathbf v$ is not null.
\end{claim}
Suppose $\mathbf v$ is null. Then the maximally extended geodesic with left-hand endpoint $q$ and initial condition $(q,\mathbf v)$ leaves $M_a$ by hypothesis. By continuous dependence of geodesics on initial conditions, $\alpha |[t_i, b)$ leaves $M_a$ for $i$ sufficiently large. This contradiction proves the claim.

\begin{claim}\label{cl:infinite-parameter}
$\alpha$ is defined on $[0,\infty)$.
\end{claim}
Suppose $b<\infty$.
Since $\mathbf v$ is not null, as $i$ increases the $(\alpha(t_i),\mathbf v_i)$ lie in compact neighborhoods of $(q,\mathbf v)$ whose intersection is $(q,\mathbf v)$. Then the existence of normal coordinate neighborhoods guarantees that $\alpha$ is the only geodesic that approaches $q$ with bounded affine parameter. Therefore $\alpha$ extends to $\alpha(b)=q$, contradicting maximality.

\begin{claim}\label{cl:constant}
There is a complete non-constant geodesic $\gamma$ such that $f\circ\gamma $ is constant.
\end{claim}

Since $f\circ\alpha$ is convex and bounded above and below, $f\circ\alpha$ is nonincreasing and $$\lim_{t\to\infty}(f\circ\alpha)(t)=c $$ for some $c\le a$.

Choose a sequence $\varepsilon_i\to 0^+$. Let the sequence $t_i$ as above increase to $\infty$ so that 
$c \le f\circ\alpha_{i}\le c+\varepsilon_i$, where $\alpha_{i}:[0,\infty)\to M$ is the geodesic with $\alpha_i(0)=\alpha(t_i)$, $\alpha_i'(0)=\mathbf v_i$. By continuous dependence of geodesics on initial conditions, the geodesic $\gamma:[0,\infty)\to M$ with $\gamma(0)=q$, $\gamma'(0)= \mathbf v$ is defined on $[0,\infty)$ and satisfies $f\circ\gamma= c$. Moreover, we can choose $t_i-i>t_{i-1}$ for each $i$. 
By claim \ref{cl:non-null}, $dt/ds_i$ is bounded above, where $t$ is the parameter of $\alpha$ and $s_i$ is the parameter of $\alpha_i$. It follows that $\gamma$ extends to $(-\infty,\infty)$ with $f\circ\gamma= c$. 
 
This completes the proof of the Lemma.
\end{proof}

\begin{lem}\label{lem:transverse}
Let $f:M\to \R$ be a 
convex function on a semi-Riemannian manifold $M$. 
Suppose a geodesic $\gamma$ satisfies\, $\gamma|[t_0-\varepsilon, t_0)\subset M_a - L_a $,\, $\gamma(t_0)\in L_a$, where $a$ is a non-critical value of $f$. 
Then $\gamma$ intersects $L_a$ transversely at $\gamma(t_0)$. 
\end{lem}

\begin{proof}
Since $ f \circ \gamma $ is convex and $ (f \circ \gamma) (t_0-\varepsilon) < (f \circ \gamma) (t_0)$, then $(f \circ \gamma)'(t_0)>0$. Hence $\gamma'(t_0)$ is not tangent to the level set $L_a$. 
\end{proof}

\begin{lem}\label{lem:connected-levels}
Let $M$ be a disprisoning semi-Riemannian manifold and $f:M\to \R$ be a nonnegative and non-constant proper convex function. 
Suppose the critical set $C$ of $f$ is connected, so $C$ is the minimum set, say $f|C=0$.
For $a>0$ and $p\in M_a-L_a$, consider
the map $\psi_{(p,a)}:S_p\,M\to L_a$, where $\psi_{(p,a)}(\mathbf v)$ is the first point at which the geodesic of $M$ with initial velocity $\mathbf v$ leaves $M_a$. Then:
\begin{enumerate}
\item\label{continuous-escape}
 $\psi_{(p,a)}$ is continous;
 \item\label{continuous-escape1}
 $\psi_{(p,a)}$ varies continously with $p \in M_a-L_a$;
 \item\label{connected-L}
 $L_a$ is connected.
 \end{enumerate}
\end{lem}

\begin{proof}
\setcounter{claim}{0}
\begin{claim}\label{cl:connected-sublevel}
$M_a$ is connected.
\end{claim}

If $[a,b]$ contains no critical value of $f$ then $M_a$ is a deformation retract of $M_b$, where for any choice of Riemannian metric $g_{\text{\,Riem}} $ on $M$, the retraction map may be taken along downward gradient curves of $f$ reparameterized by values of $f$. Thus for any $a>0$, $M_a$ is a deformation retract of $M$. By properness of $f$ and connectedness of $C$, $M$ is connected. Since $M$ is connected, $M_a$ is connected.

\begin{claim}\label{cl:non-crit-homeomorphism}
The level sets $L_a$, $a>0$, are continuously diffeomorphic. Specifically, fix $\hat a >0$. 
Then there is a diffeomorphism $\varphi:M-C\to (0,\infty)\times L_{\hat a}$ such that $\varphi|L_a$ is a diffeomorphism onto $\{a\}\times L_{\hat a}$.
\end{claim}

By disprisonment and properness, $M$ is noncompact and $f$ is unbounded. For any choice of Riemannian metric $g_{\text{\,Riem}} $ on $M$, the map $\varphi$ is given by the downward gradient curves of $f$,
reparametrized by the value of $f$.

\begin{claim}\label{cl:connected-sublevel1}
$M_a-L_a$ is connected.
\end{claim}
Immediate from claims \ref {cl:connected-sublevel} and \ref{cl:non-crit-homeomorphism}. 

\vspace{2mm}

(\ref{continuous-escape}) and (\ref{continuous-escape1}) in our lemma statement are consequences of Lemma \ref{lem:transverse}, which implies that for geodesics whose initial directions in $S_pM$ converge to that of $\gamma$, the parameter value of first departure from $M_a$ also converges to that of $\gamma$.

To prove (\ref{connected-L}), suppose a geodesic $\gamma$ from $p\in M_a-L_a$ first leaves $M_a$ by intersecting the component $L'_a$ of $L_a$. 
By (\ref{continuous-escape}), the directions of geodesics that first leave $M_a$ by intersecting $L'_a$ form a nonempty open and closed subset of $S_pM$. Thus every geodesic from $p$ first leaves $M_a$ by intersecting $L'_a$.

By (\ref{continuous-escape1}), the points $p\in M_a-L_a$ from which geodesics from $p$ first leave $M_a$ by intersecting $L'_a$ form an open and closed subset of $M_a-L_a$, hence all of $M_a-L_a$ by claim \ref{cl:connected-sublevel1}. 

There can be no component $L''_{a} \ne L'_a$ of $L_a$. Indeed, a point of $L''_{a}$ would have a normal coordinate neighborhood $U$ in $M$ such that $L_a\cap U$ lies in $L''_{a}\cap U$. Hence there would be geodesics from points in $M_a-L_a$ that first leave $M_a$ by intersecting $L''_{a}$, a contradiction. \end{proof}

The following technical lemma will be used to prove our theorems on geodesic connectedness, in particular Theorems \ref{thm:gordon-semi} and \ref{thm:main}.
 In both cases, the conditions on $C$ will be easily verified. 

\begin{lem}\label{lem:semi-geo-conn}
Let $M$ be a disprisoning semi-Riemannian manifold and $f:M\to \R$ be a nonnegative proper convex function. Suppose any two points of the critical set $C$ of $f$ are joined by a geodesic of $M$, and $C$ has an oriented neighborhood $U$ in $M$. Then $M$ is oriented.

Suppose further that for some non-critical value $a$, there is $p\in M_a-L_a$ such that the map $\psi_{(p,a)}:S_p\,M\to L_a$ has nonzero degree, where 
$\psi_{(p,a)}(\mathbf v)$ is the first point at which the geodesic of $M$ with initial velocity $\mathbf v$ leaves $M_a$. 
 Then $M$ is geodesically connected.
\end{lem}

\begin{proof}
\setcounter{claim}{0}
By convexity of $f$, critical points of $f$ are local minima, $C$ is geodesically connected, and $C$ is the minimum set of $f$, say $f|C=0$. 

\begin{claim}\label{cl:orientable-levels}
$M$ and the level sets $L_a$, $a>0$, have an orientation determined by the given oriented neighborhood $U$ of $C$.
\end{claim}

Downward gradient flow of $f$ in $\,g_{\text{\,Riem}} $ carries a coordinate neighborhood of each point in $M$ diffeomorphically into $U$. Hence coordinates on $M$ may be chosen so that all transition functions have positive determinant. This orientation of $M$ induces an orientation on each $L_a$ by requiring the coordinate basis of 
 $T_pL_a$, followed by the gradient vector in $\,g_{\text{\,Riem}} $ of $f$ at $p$, to be a positively oriented basis of $T_pM$.

\begin{claim}\label{cl:constant-degree}
For any $a>0$, the degree of $\psi_{(p,a)}$ is constant for all $p\in M_a-L_a$.
\end{claim} 

The level set $L_a$ is compact by properness of $f$, connected by Lemma \ref{lem:connected-levels}, and oriented by claim \ref{cl:orientable-levels}. Thus degree of $\psi_{(p,a)}$ is defined.

By claim \ref{cl:connected-sublevel1} of Lemma \ref{lem:connected-levels}, two points of $M_a-L_a$ are joined by a path $\alpha:[0,1]\to M_a-L_a$. Parallel translation along $\alpha$ with respect to $\,g_{\text{\,Riem}}$ identifies the pull-back bundle $\alpha^*(SM )$ of $SM$ along $\alpha$ homeomorphically with $S^{n-1}\times [0,1]$. By 
disprisonment and Lemma \ref{lem:connected-levels}, the maps $\psi_{(\alpha(t),a)}:S_{\alpha(t)}M\to L_a$ determine a one-parameter family of continuous maps $S^{n-1}\to L_a$. Since these maps vary continuously in $t$, their degree is constant. 

\begin{claim}\label{cl:geo-conn-<a}
For any $a>0$ and every $p\in M_a-L_a$, there is a geodesic from $p$ to every point of $L_a$.
\end{claim}
 By claim \ref{cl:constant-degree} of Lemma \ref{lem:connected-levels}, the map
$$(\varphi |L_a)\circ\psi_{(p,a)}:S_pM\to \{a\}\times L_ {\hat a}$$
has constant degree for all $p\in M_a-L_a$. Moreover, this map
varies continously in $a$, and so has constant degree for all $a>0$ and all $p\in M_a-L_a$. Then the claim follows from our degree hypothesis.

\begin{claim}\label{cl:geo-conn-=a}
For any $a>0$ and every $p\in M_a$, there is a geodesic from $p$ to every point of $L_a$.
\end{claim} 
By claim \ref{cl:geo-conn-<a}, it suffices to show that any two distinct points $p,q\in L_a$ are joined by a geodesic. 
In particular, for $p_i\to p$, $p_i\in M_a-L_a$, there is a geodesic $\gamma_i\,:\,[0,c_i]\to M_a$ satisfying $\gamma_i(0)=p_i$, $\gamma_i(c_i)=q$, $(p_i,\gamma_i'(0))\in SM$.
Since $M_a$ is compact, we may assume 
the sequence $(p_i,\gamma_i'(0))$ converges in $SM$ to $(p,\mathbf v)\in SM$, $\mathbf v\ne 0$. We may also assume $c_i\to c\in (0,\infty]$. 

Let $\gamma$ be the maximally extended geodesic with $\gamma(0)= p$, $\gamma'(0)=\mathbf v$. If $c<\infty$, then $\gamma\,|\,(0,c]$\, is defined and joins $p$ to $q$. If $c=\infty$, then $\gamma$ is defined on $[0,\infty)$ and does not leave $M_a$, a contradiction since $M$ is disprisoning.

 \begin{claim}\label{cl:geo-conn}
For every $p\in M_a$, there is a geodesic from $p$ to every $q\in M_a$.
\end{claim}
If $p,q\in C$, the claim is true by hypothesis. Otherwise, set $$a=\max\{f(p), f(q)\}.$$
Then claim \ref{cl:geo-conn} follows from claims \ref{cl:geo-conn-<a} and \ref{cl:geo-conn-=a}. 

\vspace{2mm}
Hence the lemma.
\end{proof}

 \emph{Proof of Theorem \ref{thm:gordon-semi}.}  
Since the critical set is $C=\{p\}$, $C$ has an orientable neighborhood. By Lemmas \ref{lem:exit-or-level} and \ref{lem:semi-geo-conn}, it suffices now to observe that if $a$ is sufficiently small, $\psi_{(p,a)}$ maps $S_p\,M$ diffeomorphically onto $L_a$ and hence has degree one.
\qed

%%%%%%%%%%%%%%%%%%%%%%%%%%%%%%%%%%%%%%%%%%%%%%%%%%%%%%%%%%%%%%%%%%%%%%
 
\section {Graphs and geodesic connectedness}
Let us briefly review some basic Lorentzian terminology.

Suppose $(M,\langle \cdot, \cdot \rangle)$ is a Lorentzian manifold. A nonzero tangent vector $\mathbf v$ is \emph{timelike, spacelike, non-spacelike} or \emph{null} according to whether $\langle \mathbf v,\mathbf v\rangle$ is negative, positive, non-positive or zero respectively. For each $p \in M$ the set of all non-spacelike vectors in $T_p M$ consists of two connected components, that may be called \emph{hemicones}. A continuous choice of hemicone for all $p \in M$ is called a \emph{time orientation} of $M$. 
A Lorentzian manifold with a choice of time orientation is called a \emph{space-time}.

In a space-time, the vectors in the chosen hemicones are called \emph{future-pointing}. 
For two points $p,q \in M$ we write $p \leqslant q$ if $p = q$ or if there is a piecewise smooth curve with future-pointing 
(possibly one-sided) 
tangent vectors from $p$ to $q$. The \emph{causal future} of $p\in M$ is $J^+_M(p) = \{ q \in M : p \leqslant q \}$ and the \emph{causal past} is $J^-_M(p) = \{ q \in M : q \leqslant p \}$. 

%Following \cite{bee}, a continuous curve $\gamma : (a,b) \to M$ is said to be a \emph{future-directed non-spacelike curve} if for each $t_0 \in (a,b)$ there is an $\varepsilon > 0$ and a convex normal neighborhood $U(\gamma(t_0))$ of $t_0$ with $\gamma(t_0 - \varepsilon, t_0 + \varepsilon) \subseteq U(\gamma(t_0))$ such that given any $t_1, t_2$ with $ t_0 - \varepsilon < t_1 < t_2 < t_0 + \varepsilon$, there is a smooth future-directed non-spacelike curve in $U(\gamma(t_0))$ from $\gamma(t_0)$ to $\gamma(t_1)$.

An open neighborhood in a space-time $M$ is \emph{causally convex} if every piecewise smooth curve with future-pointing tangent vectors intersects it in a connected set. A space-time is \emph{strongly causal} if every point has arbitrarily small causally convex neighborhoods. A strongly causal spacetime is \emph{globally hyperbolic} if $J^+_M(p) \cap J^-_M(q)$ is compact for all $p,q \in M$.

A \emph{Cauchy hypersurface} $\Sigma$ is a 
hypersurface that 
is intersected by every inextendible causal curve exactly once. 
A space-time is globally hyperbolic if and only if it admits a Cauchy hypersurface
 \cite[p. 211]{he}.

In this section, we prove Corollary \ref{cor:lorentz-graph} on geodesic connectedness of graphs of strictly convex functions.

\begin{lem}\label{lem:convex-lift}
Suppose $M$ is a Lorentzian manifold and $ u : M \to \R $ is a Lorentzian function. Let $\Gamma(u) $ be the graph of $u$ in $ M \times \R $. Let the function $ f : \Gamma(u) \to \R $ be the lift of $u$, defined by $ f(p, u(p)) = u(p) $. Then for any vectors $\overline{\mathbf x}, \overline{\mathbf y} \in T_{(p, u(p))} \Gamma(u)$, with corresponding vectors $\mathbf x, \mathbf y \in T_p M$ obtained by projection onto $M$, \begin{equation}\label{lem:hess-of-lift}
\hess f\left( \overline{\mathbf x}, \overline{\mathbf y} \right) = \frac{ \hess u(\mathbf x,\mathbf y) }{1 + \langle (\del u)_p, (\del u)_p \rangle }.
\end{equation}
\end{lem}
 
\begin{proof} 
Suppose $ \gamma(t) = (\alpha(t),u(\alpha(t))) $ is a geodesic of $\Gamma(u)$. Then the second covariant derivatives satisfy
$$
\gamma''(t) = \alpha''(t) + (u \circ \alpha)''(t)\, \partial _y|_{(u \circ \alpha) (t)}
 $$ 
 where $ \partial_y $ is the standard coordinate vector field on the second factor of $ M \times \R $.
 
 Any vector field $Y$ on $\Gamma(u)$ can be written as $ Y = X + (Xu)\,\partial_y = X + \langle \del u, X \rangle\,\partial_y $ where $ X $ is a vector field on $M$. In order for $\gamma$ to be a geodesic, $ \gamma''(t) $ must be orthogonal to $\Gamma(u)$ in $ M \times \R $. Thus 
 $$
\langle \gamma''(t) , Y_{\gamma(t)} \rangle = \langle \alpha''(t), X_{\alpha(t)} \rangle + (u \circ \alpha)''(t)\langle (\del u)_{\alpha(t)}, X_{\alpha(t)} \rangle 
$$ $$
= \langle \alpha''(t) + (u \circ \alpha)''(t) (\del u)_{\alpha(t)}, X_{\alpha(t)} \rangle = 0 
 $$ 
 for any vector field $X$ on $M$, so \begin{equation}\label{geo-in-graph}
 \alpha''(t) + (u \circ \alpha)''(t) (\del u)_{\alpha(t)} = \mathbf 0. 
\end{equation} 
 
Therefore $$ (u \circ \alpha)''(t) = \hess u( \alpha'(t), \alpha'(t) ) + \langle (\del u)_{\alpha(t)}, \alpha''(t) \rangle 
$$ $$= \hess u( \alpha'(t), \alpha'(t) ) - (u \circ \alpha)''(t) \langle (\del u)_{\alpha(t)}, (\del u)_{\alpha(t)} \rangle .
$$ Moreover, $$ \hess f( \gamma'(t), \gamma'(t) ) = (f \circ \gamma)''(t) = (u \circ \alpha)''(t) = \frac{\hess u( \alpha'(t), \alpha'(t) )}{ 1 + \langle (\del u)_{\alpha(t)}, (\del u)_{\alpha(t)} \rangle }.$$ 
Since this holds for any geodesic and $ \alpha'(t) $ is the projection of $ \gamma'(t) $ to $T_{\alpha(t)}M$, we conclude that for any tangent vector $\overline{\mathbf x} \in T_{(p, u(p))} \Gamma(u)$, $$ \hess f \left( \overline{\mathbf x}, \overline{\mathbf x} \right) = \frac{ \hess u(\mathbf x,\mathbf x) }{1 + \langle (\del u)_p, (\del u)_p \rangle } $$ where $ \mathbf x $ is the projection of $\overline{\mathbf x}$ onto $T_p M$. Equation (\ref{lem:hess-of-lift}) follows since symmetric bilinear forms on vector spaces are determined by their corresponding quadratic forms. 
\end{proof}

\begin{lem}\label{lem:strong-causal-subman} 

If $M$ is a strongly causal space-time and $H$ is a Riemannian manifold, then any 
immersed timelike submanifold $E$ of $M \times H$ is a strongly causal space-time
in the induced Lorentzian metric. 
\end{lem}

\begin{proof}
By \cite[Lemma 3.54 and Proposition 3.62]{bee}, $M \times H$ is a strongly causal space-time since $M$ is a strongly causal space-time.
Since the timelike tangent vectors to $E$ form the intersection of $TE$ with the timelike vectors in the pull-back of $T(M \times H)$, $E$ inherits a time orientation from $M \times H$. 

Suppose $E$ is not strongly causal.
For $p\in E$, every sufficiently small neighborhood $\tilde U$ of $p$ in $M \times H$ lies in a coordinate neighborhood whose intersection with $E$ is a coordinate slice.
Moreover, there is a piecewise smooth curve $\gamma$ with future-pointing tangent vectors in the component $U'$ of $E\cap\tilde U$ containing $p$, such that $\gamma$ intersects $U'$ in a disconnected set. But then $\gamma$ intersects $\tilde U$ in a disconnected set. This contradiction shows $E$ inherits strong causality from $M \times H$. 
 \end{proof}

 \emph{Proof of Corollary  \ref{cor:lorentz-graph}}. 
Since the graph $\Gamma(u)$ is a timelike submanifold of $M \times \R$, then $\Gamma(u)$ is a strongly causal space-time by Lemma \ref{lem:strong-causal-subman}. Therefore $\Gamma(u)$ is null-disprisoning \cite[Proposition 3.13]{bee}.

Let $f : \Gamma(u) \to \R$ be the lift of $u : M \to \R$ to the graph $ \Gamma(u) $. Since $u$ is proper and nonnegative, so is $f$. In addition, $f$ is strictly convex by Lemma \ref{lem:convex-lift} since $u$ is strictly convex. By Theorem 
\ref{thm:gordon-semi}, $\,\Gamma(u)$ is geodesically connected.
\qed

%%%%%%%%%%%%%%%%
%%%%%%%%%%%%%%%%%
%%%%%%%%%%%%%%%%%
%%%%%%%%%%%%%%%%%%%%%%%%%%%%%%%%%%%%%%%%%%%%%%%%%%%%%%%%%%%%%%%%%%%%%%%%%%%%%%%%%%%%%

\section{Dual cones in semi-Euclidean space}
\label{sec:dualcones}

In order to generate proper convex functions on timelike convex hypersurfaces in Minkowski space, we need to extend the notion of dual cones in Euclidean space to semi-Euclidean and Minkowski space. In Section \ref{sec:hypersurface} we apply the theory to convex hypersurfaces.

\begin{definition}
\label{def_dualcone} Let $K$ be a subset of $\SEu $. The \emph{dual cone} $K^*$ of $K$ in $\SEu $ is defined by $K^* =\{ \x^* \in \SEu : \langle \x^*, \x \rangle \geq 0 \text{ for all } \x \in K \}.$
\end{definition}

\begin{prop} 
\label{dual_props} Let $K, K_1 \subseteq \SEu$.
 \begin{enumerate}
\item $ K^* $ is a closed convex cone.
\item If $ K_1 \subseteq K $, then $ K_1^{*} \supseteq K^* $.
\item $ (-K)^* = - K^* $.
\item If $K$ has nonempty interior relative to $\SEu$, then $K^*$ is 
pointed, i.e. $K^*$ contains no line.
\item $ K^{**} $ is the closure of the smallest convex cone containing $K$.
\item Let $ \partial K $ denote the boundary of $K$ relative to $\SEu$. If $K$ is a convex cone, then $ \x \in \partial K $ if and only if $ \langle \x, \x^* \rangle = 0 $ for some $ \x^* \in K^* $.
\end{enumerate} 
\end{prop}

\begin{proof}
\emph{(1)--(5).} Given a subset $ K $ of a finite dimensional vector space $V$, one can define the dual cone in the dual space as the linear functionals $\ell$ on $V$ with $\ell(\mathbf x) \geq 0$ for $\mathbf x \in K$. These properties are well-known properties of this dual cone.

Equipping $V$ with a nondegenerate symmetric bilinear form $ \langle \cdot, \cdot \rangle $ identifies $V$ with its dual space. All linear functionals can be represented as $ \ell_{\w}( \x ) = \langle \w, \x \rangle $ for $ \w, \x \in V $. In this representation, the dual cone of a subset $ K \subseteq V $ is $ K^* = \{ \x^* \in V : \langle \x^*, \x \rangle \geq 0 \} $. Thus the same properties are carried over to the dual cone defined using the inner product, in particular if we take the semi-Euclidean inner product on $\SEu$.

\emph{(6).} Suppose $ \x_0 \in K - \partial K$. Then $ \langle \x_0 + \mathbf u, \x^* \rangle \geq 0 $ for all $ \mathbf u \in U $ for some neighborhood $ U $ of $0$ in $\SEu$ and for any $ \x^* \in K^* $. For any $ \x^* \in K^* $, 
choose $ \mathbf u \in U $, $ \langle \mathbf u, \x^* \rangle < 0 $. Then $ \langle \x_0, \x^* \rangle > 0 $ for any $ \x^* \in K^* $.

On the other hand, suppose $ \x_0 \in \partial K $. Take $ \x^* \in \SEu $ to be a nonzero normal vector to a supporting hyperplane of $K$ at $\x_0 \in K$ with $ \langle \x - \x_0, \x^* \rangle \geq 0 $ for all $ \x \in K $. Letting $ \x = \lambda \x_0 $ for any $ \lambda >0$, then $ (\lambda - 1) \langle \x_0, \x^* \rangle \geq 0 $. Letting $ \lambda > 1 $ and $ \lambda < 1 $ we obtain $ \langle \x_0, \x^* \rangle = 0 $. The claim follows.
\end{proof}

\begin{prop}
\label{future_past_dual} Let $ \Future $ and $\Past = -\Future$ denote the closed future and past cones in $\M$, respectively. Then $\Future^* = \Past$ and $ \Past^* = \Future$.
\end{prop}

\begin{proof}
A simple calculation shows that $ \mathbf u \in \Future $ if and only if $ \langle \mathbf u, \w \rangle \geq 0$ for all $ \w \in \Past $.
\end{proof}

\begin{lem}
\label{lem:max_cone}
Let $ K \subseteq \M $ be a spacelike convex cone. Then either $K$ is contained in a subspace of $\M$ of dimension $\leq n$ or $K$ has nonempty interior relative to $\M$ and $K^*$ contains a pair of linearly independent null vectors, $ \mathbf u \in \Future $ and $\mathbf u' \in \Past $.
\end{lem}

\begin{proof}
Since $K$ and $\Future$ are convex and and intersect only at the origin, we can find a separating hyperplane $ H_{\w} = \{ \x \in \M : \langle \x, \w \rangle = 0 \} $, $ \w \neq \mathbf{0} $, such that $ \w \in K^*$ and $ -\w \in \Future^* = \Past$, i.e. $\w \in \Future $. Thus $ \w \in \Future \cap K^* $. Similarly, we can find a nonzero vector $ \w' \in \Past \cap K^* $.

If $ \w $ and $ \w' $ are scalar multiples of one another, then $ K^* $ contains a line, and $ K $ lies in a subspace of dimension $\le n$. Otherwise $ \w $ and $ \w' $ are linearly independent. Then the line segment between $ \w $ and $ \w' $ passes through a pair of linearly independent null vectors $ \mathbf u $ and $ \mathbf u' $, future-oriented and past-oriented respectively. Since $K^*$ is convex, $ \mathbf u, \mathbf u' \in K^* $. 
\end{proof}

\section{Convex hypersurfaces and geodesic connectedness} 
\label{sec:hypersurface}

In this section, we prove Theorem \ref{thm:main} on geodesic connectedness of a timelike convex hypersurface $M$. The method is by constructing a convex function on $M$.

First we show that $M$ is essentially the graph of a convex function over at least one of its tangent hyperplanes (Theorem \ref{thm:wu2}). 
Wu proved the analogous theorem for Euclidean convex hypersurfaces in \cite{wu}. In the Minkowski setting, the argument is somewhat more delicate (see Lemma \ref{lem:RN} and Example \ref{RN-counterexample}).

The proof depends on Lemma \ref{lem:RN} concerning
the normal and recession cones of $M$. We begin with a few lemmas on normal and recession cones of 
general
convex hypersurfaces in semi-Euclidean space. By a 
\emph{general convex hypersurface} we will mean the boundary of a convex body, not necessarily smooth and not necessarily connected. (The latter provision merely allows the possibility of two parallel hyperplanes).

Unless otherwise specified, ``interior'' and the symbol ``$\inte$'' will mean interior relative to the original ambient semi-Euclidean space.

\begin{definition}
\label{def_Recess} Let $M$ be a 
general
convex hypersurface of $\SEu$ bounding the convex body $B$ in $\SEu$. 
\begin{enumerate}
\item
The \emph{recession cone} $\Recess$ of $M$ consists of all vectors on any ray from $\mathbf 0$ in $\SEu$ that is the translate of a ray in $B$. 
\item
The \emph{normal cone} $\Normals$ of $M$ consists of all nonzero vectors $\w \in \SEu$ such that the halfspace $ \{ \x \in \SEu : \langle \x, \w \rangle \geq 0 \} $ is a translate of a 
supporting halfspace of $B$ at some $p \in M$, i.e. a halfspace that contains $B$ and whose boundary is tangent to $B$ at $p$.
\end{enumerate}
\end{definition}

\begin{definition}
Given a choice of orthonormal basis of $\SEu$, the \emph{associated Euclidean space} $ \E^{n+k} $ is obtained by making the basis Euclidean orthonormal.
\end{definition}
 
\begin{remark}\label{rem:associated-orthog}
$ \mathbf w = (w^1, ... , w^{n+k}) $ is orthogonal to $\mathbf w_0$ in $\SEu$ if and only if $ \mathbf w' = (w^1, ... , w^{n}, -w^{n+1}, ... , -w^{n+k}) $ is orthogonal to $\mathbf w_0$ in the associated Euclidean space.
\end{remark}

\begin{lem}\label{thm:wu1}
Let $M$ be a general convex hypersurface in $ \SEu $, and $\Normals $ be the normal cone of $M$. Then there exist a unique subspace $ V \subseteq \SEu $ and a unique open convex cone $K$ in $ V $ such that $ K \subseteq \Normals \subseteq \overline{K} $, i.e. the closure and the interior relative to $V$ of $\Normals$ are convex.
\end{lem}

\begin{proof}
Regard $M$ as a convex hypersurface in an associated Euclidean space $\E^{n+k}$, and let $ N : M \to S^{n+k-1} $ denote the Gauss map in $\E^{n+k}$. By Theorem 1 in \cite{wu}, there exist a unique totally geodesic sphere $ S^m \subseteq S^{n+k-1} $ and a unique open convex subset $U$ of $S^m$ such that $ U \subseteq N(M) \subseteq \overline{U} $.

For a set $W$ in a vector space $V$, we set $\text{ray} \, W = \{ \lambda \mathbf w : \mathbf w \in W, \lambda \in [0,\infty) \} $. In $ \E^{n+k} $, there is a one-to-one correspondence between open (closed) convex subsets of the unit sphere and open (closed) convex cones, obtained by identifying a point on the sphere with the open (closed) ray from the origin through that point. Thus there exist a unique subspace $ V = \text{ray} \, S^m $ in $\E^{n+k}$ and a unique open convex cone $K' = \text{ray} \, U - \{ \mathbf 0\} \subseteq V$ such that $ K'\subseteq \text{ray} \, N(M) \subseteq \overline{K'} $. 

Since $ \mathbf w = (w^1, ... , w^{n+k}) $ is an inward normal vector to $M$ at a point $p$ in $\SEu$ if and only if $ \mathbf w' = (w^1, ... , w^{n}, -w^{n+1}, ... , -w^{n+k}) $ is an inward normal vector to $M$ at $p$ in the associated Euclidean space $ \E^{n+k} $, we have a vector space isomorphism mapping the normal cone $ \Normals_{\text{Euc}} = \text{ray} \, N(M) - \{ \mathbf 0 \}$ in the associated Euclidean space to the normal cone $ \Normals $ in $ \SEu $. All convex sets are carried to convex sets and the theorem follows.
\end{proof}

\begin{remark}
By Lemma \ref{thm:wu1}, the normal cone of a general convex hypersurface has convex interior relative to the subspace $V$, and convex closure. In \cite{wu}, an example of a smooth convex hypersurface $M$ in $\E^3$ is described to show that the normal cone itself need not be convex. 

To construct an analogous example in $ \SEu$, consider an associated Euclidean space $ \E^{n+k} $ and a convex hypersurface $M$ in $ \E^{n+k} $ whose normal cone $\Normals_{\text{Euc}}$ is not convex. Since there is a vector space isomorphism $\E^{n+k} \to \SEu$ mapping the normal cone in the associated Euclidean space to the normal cone in $\SEu$, the normal cone $\Normals$ in $\SEu$ will not be convex.
\end{remark}

\begin{lem}
\label{lem:dual}
Let $M$ be a general convex hypersurface of $\SEu$ with recession cone $\Recess$ and normal cone $\Normals$. Then $\Normals^* = \Recess$ and $ \Recess^* = \overline{\Normals} $.
\end{lem}

\begin{proof}
Let $B$ be the convex body bounded by $M$. Suppose $ \mathbf u \in \Recess$. Let $ \w \in \Normals $ be a nonzero normal vector at $\mathbf a \in M$. Then $ \x=\mathbf u + \mathbf a \in B $. By definition of $\Normals$, the hyperplane orthogonal to $ \w$ supports $B$ at $ \mathbf a $, i.e. $ \langle \mathbf y - \mathbf a, \w \rangle \geq 0 $ for all $ \mathbf y \in B $. In particular, $ \langle \x - \mathbf a, \w \rangle = \langle \mathbf u,\w \rangle \geq 0$.

Suppose $ \mathbf u \not\in \Recess $. Choose $ \x \in \inte B $. The ray $ \{ \x + \lambda \mathbf u : \lambda > 0 \}$ leaves $B$, say at $\mathbf a = \x + \lambda \mathbf u \in M$ for some $\lambda > 0$. Let $ \w \in \Normals$ be a nonzero normal to $M$ at $ \mathbf a$. Then $ \langle \x - \mathbf a, \w \rangle > 0 $ since $ \x \in \inte B $, so $ \langle -\lambda \mathbf u , \w \rangle > 0 $ and $ \langle \mathbf u , \w \rangle < 0 $. Thus by Definition \ref{def_dualcone}, $\Normals^* = \Recess$.

By Proposition \ref{dual_props}, $ \Recess^* = \Normals^{**} $ is the closure of the smallest convex cone containing $\Normals$. By Lemma \ref{thm:wu1}, $\overline{\Normals}$ is the closure of the smallest convex cone containing $\Normals$, so $ \Recess^* = \overline{\Normals} $.
\end{proof}

%\begin{remark}\label{remark:convex-hyp}
%The smoothings in Example \ref{ex:race-track} may be carried out without changing the recession cone, hence by Lemma \ref{lem:dual}, without changing $ \Recess^* = \overline{\Normals} $. Since $\overline{\Normals} $ is timelike, it follows that the smoothed hypersurfaces are timelike.
% \end{remark}

Now we return to smooth convex hypersurfaces (Definition \ref{def:convex-hyp}).
We will use a definition of strong strict convexity that makes sense even at degenerate points, where second fundamental form is undefined:

\begin{definition}\label{def:weak-strict-convexity} Let $M$ be a convex hypersurface of $\SEu$. We say $p\in M$ is a \emph{point of weak strict convexity} if $$M\cap T_pM=\{p\}.$$ 
We say $p\in M$ is a \emph{point of strong strict convexity} if a neighborhood of $p$ in $M$ is the level set of a regular function that is defined on a neighborhood of $p$ in $\SEu$ and has definite Hessian on $T_p M$.
(Equivalently, $p$ is a point of strong strict convexity in an associated Euclidean space.)
\end{definition}

In light of the following lemma, we may speak of convex hypersurfaces \emph{with a point of strict convexity} without specifying the type:

\begin{lem}\label{lem:weak-strong}
Let $M$ be a 
convex hypersurface of\, $\SEu$. Then the following are equivalent:
\begin{enumerate}
\item\label{strong}
 $M$ has a point of strong strict convexity, 
 \item\label{no-line}
 $M$ contains no line of $\SEu$,
 \item\label{weak}
 $M$ has a point of weak strict convexity.
 \end{enumerate}
\end{lem}

\begin{proof}
$(\ref{no-line})\Rightarrow (\ref{strong})$:
By Lemma 2 of \cite{h-n} or Lemma 2 of \cite{c-l}, applied to the embedding of $M$ in an associated Euclidean space $\E^{n+k}$, if there is no point of strong strict convexity then $M$ contains a line.

$(\ref{weak})\Rightarrow(\ref{no-line})$: If $M$ contains a line, then $M$ is ruled by parallel lines. Therefore $M$ contains no point of weak strict convexity.

$(\ref{strong})\Rightarrow (\ref{weak})$: Obvious.
\end{proof}

\begin{lem}\label{lem:RN}
Suppose $M$ is a timelike convex hypersurface in $\M$, $n\ge 2$, bounding the convex body $B$, and having a point of strict convexity. Let $\Recess$ and $\Normals$ denote the recession cone and normal cone of $M$, respectively. Then there is a nonzero vector $ \mathbf v _0 \in (\inte \Normals) \cap \Recess $.
\end{lem}

\begin{proof}
Since $M$ has a point of strong strict convexity by Lemma \ref{lem:weak-strong}, then $ \inte \Normals \ne\emptyset$.

Suppose $(\inte \Normals) \cap \Recess = \emptyset $. Then the convex cones $ \inte \Normals$ and $\Recess$ are separated, i.e. lie in opposite closed halfspaces bounded by some $n$-dimensional subspace $H$. Let $\w$ be a nonzero normal vector to $H$, chosen so that $\langle \w, \mathbf n \rangle \leq 0$ for all $\mathbf n \in\Normals$ and $ \langle \w, \mathbf u \rangle \geq 0 $ for all $ \mathbf u \in \Recess $. Then $\w\in -\Recess$ and $\w \in \Recess^* = \overline{\Normals} $ by Lemma \ref{lem:dual}. Since $\Normals$ consists of spacelike vectors (because $M$ is timelike), $ \langle \w, \w \rangle \geq 0 $, but since $ \w \in -\Recess$, $ \langle \w, \w \rangle \leq 0 $. We conclude that $ \w $ is null, so $ \w \in \Future \cup \Past $.
 
Since $\inte\Normals$ is a spacelike convex cone with nonempty interior relative to $\M$, we can apply Lemma \ref{lem:max_cone} and choose a pair of linearly independent null vectors $ \mathbf u \in \Future \cap \Recess $ and $ \mathbf u' \in \Past \cap \Recess $. Since $\Recess = \Normals^* = \overline{\Normals}^*$, $ \langle \w, \mathbf u \rangle \geq 0 $ and $ \langle \w, \mathbf u' \rangle \geq 0 $. However, this means that $ \w \in \Future \cap \Past = \{ \mathbf 0 \} $, a contradiction.
\end{proof}

The following example shows that, in contrast to $\Eu$ \cite{wu}, for a 
non-timelike convex hypersurface in $\M$ with $\inte \Normals \neq \emptyset$, $ (\inte \Normals) \cap \Recess $ can be empty.

\begin{example}\label{RN-counterexample}
Let $B = \{ (x, t) \in \mink : xt \geq 1, x > 0 \}$ and $ M = \partial B $. $M$ is a strictly convex hypersurface in $\mink$. The interior of the normal cone $ \inte \Normals $ is the open fourth quadrant and the recession cone $ \Recess$ is the closed first quadrant, so $( \inte \Normals) \cap \Recess= \emptyset $.
\end{example}

\begin{thm}\label{thm:wu2}
Suppose $M$ is a timelike convex hypersurface in $\M$, $n\ge 2$, with a point of strict convexity. Then coordinates of $\M$ can be chosen so that the tangent hyperplane $T_{\mathbf 0}M$ to $M$ at the origin is $\{x^1=0\}$, and the following properties hold, where \,$ \inte D$ and \,$\partial D$ denote the interior and boundary of $D$ relative to \,$T_{\mathbf 0}M$:

\begin{enumerate}

\item\label{graph}
Let $ \Pi : \M \to T_{\mathbf 0} M $ be orthogonal projection, and $D$ be the convex set $ \Pi(M) $. Then over $ \inte D $, $M$ is the graph of a convex function $u: \inte D \to \R$.

\item
\label{halfline}
For every $p \in D - \inte D $, $M\cap \Pi^{-1}(p)$ is a closed spacelike half-line of $\M$.
 
 \item
 \label{level}
For any $a > 0$, $L_a = M \cap \{ x^1 = a \} $ is homeomorphic to $ S^{n-1} $
\end{enumerate}
\end{thm}

\begin{proof}
\emph{(\ref{graph}), (\ref{halfline}), (\ref{level}).} 
Choose $ \mathbf v_0 \in (\inte \Normals) \cap \Recess $ as in Lemma \ref{lem:RN}, and linear coordinates on $ \M $ so that the tangent hyperplane $T_{\mathbf 0}M$ is $\{x^1=0\}$ and $\mathbf v_0 = (1, 0, ... , 0) \in \Normals$ is an inward normal to $M$ at the origin. 
Since a compact convex hypersurface of $\M$ cannot have all tangent planes timelike, $M$ is noncompact. If we regard $M$ as embedded in the associated Euclidean space defined by these coordinates, $M$ becomes the boundary of a convex body in $\Eu$ that contains no lines of $\Eu$.
 
By choice of $\mathbf v_0$, it remains true that $ \mathbf v_0 \in (\inte \Normals) \cap \Recess $ when $\Normals$ and $\Recess$ are defined in this associated $\Eu$. By Theorem 2 in \cite{wu}, (\ref{graph}), (\ref{halfline}), and (\ref{level}) hold because orthogonal projection in the associated $\Eu$ to $T_{\mathbf 0} M$ is the same map as orthogonal projection in $\M$ to $T_{\mathbf 0} M$. \end{proof}

\begin{remark}
Although $M$ has a point of (strong or weak) strict convexity, it is not always possible to choose the coordinates so that the origin
in Theorem \ref{thm:wu2}
is such a point. Thus coordinates cannot always be chosen so that the critical set of $u$ is a point. Locating an origin depends on Lemma \ref{lem:RN} concerning $(\inte \Normals) \cap \Recess $.
\end{remark}

\begin{thm}\label{thm:support-convex}
Suppose $M$ is a timelike convex hypersurface of $\M$ with a point of strict convexity. Then $M$ supports a proper nonnegative convex function $f$. If $M$ is strongly strictly convex, then $M$ supports a proper 
strictly convex function $f$.
\end{thm}

\begin{proof}
Consider coordinates on $\M$, projection 
$ \Pi : \M \to T_{\mathbf 0} M $, and $ D = \Pi(M) $ as in Theorem \ref{thm:wu2}. Set $ f = x^1 |M = \langle \cdot, \mathbf v _0 \rangle $. 

Let $ \gamma : (a,b) \to M $ be a geodesic of $M$, and $ N : M \to \M $ be the unit normal field on $M$ with $ N_p \in \Normals $ for all $p \in M$.

The acceleration of $\gamma$ in $\M$ can be written as $ \gamma''(t) = \langle \gamma''(t), N_{\gamma(t)} \rangle N_{\gamma(t)} $, so $ (f \circ \gamma)''(t) = \langle \gamma''(t), \mathbf v_0 \rangle = \langle \gamma''(t), N_{\gamma(t)} \rangle \langle N_{\gamma(t)}, \mathbf v_0 \rangle $. Since $ M$ is convex, the acceleration must be an inward normal vector at each point along $\gamma$, in the sense that $ \langle \gamma''(t), N_{\gamma(t)} \rangle \geq 0$. Additionally, since $ \mathbf v _0 \in \Recess $ and $ N_{\gamma(t)} \in \Normals $, $\langle N_{\gamma(t)}, \mathbf v_0 \rangle \geq 0$. Thus, $f$ is convex.

If $M$ is 
strongly
strictly convex, then $\gamma''(t) \neq 0$ and $\langle \gamma''(t), N_{\gamma(t)} \rangle > 0$ along $\gamma$. Moreover, the image of the Gauss map in the associated Euclidean space, and consequently its normal cone, is open and contains none of its boundary points. If $\langle N_{\gamma(t)}, \mathbf v _0 \rangle = 0$ for some $t$, then by $(6)$ of Proposition \ref{dual_props}, $\mathbf v _0$ is in the boundary of the normal cone, a contradiction, so $ \langle N_{\gamma(t)}, \mathbf v _0 \rangle > 0 $ along $\gamma$. We conclude that if $M$ is strongly
strictly convex, then $(f \circ \gamma)''(t) > 0$ along any non-constant geodesic, i.e. $f$ is strictly convex.

Finally we show $f$ is proper. Otherwise, there is some sublevel $ M_a = \{ p \in M : f(p) \leq a \} $ that is not compact. Then $ \Pi(M_a) $ is a noncompact closed convex subset of
$D$ and has noncompact boundary $ \partial \Pi(M_a) = \Pi(\partial M_a) = \Pi(L_a) $, contradicting compactness of $L_a$. \end{proof}

Now we are ready to prove Theorem \ref{thm:main}, which states that  a timelike convex hypersurface $M$ of $\M$ is  geodesically connected if, after splitting off a semi-Euclidean factor of maximal dimension, 
$M$ contains no rolled Euclidean half-plane.
 
 The proof is divided into two cases, depending on whether the timelike convex hypersurface $M$ of $\M$ is  not or is ruled by parallel  null lines.  In the latter case, we use the following criterion of Bartolo, Candela, and Flores, the proof of which  uses infinite dimensional variational methods \cite{bcf}.

\begin{thm}\cite[Theorem 1.2]{bcf} \label{thm:null-killing}
Let $(M,\langle \cdot, \cdot \rangle_M)$ be a globally hyperbolic space-time endowed with a complete \emph{null} Killing vector field $K$ and a complete (smooth, spacelike) Cauchy hypersurface $S$. 
Then $M$ is geodesically connected if and only if 
any two points of $M$ are joined by a $C^1$ curve $\varphi$ in $M$ such that $\langle \varphi', K \circ \varphi \rangle_M$ either has constant sign or vanishes identically.
\end{thm}
 
\vspace*{3mm}

\noindent\emph{Proof of  Theorem \ref{thm:main}.} 
Let $k$ be the maximal dimension of a nondegenerate $k$-plane $P$ contained in $M$. Then $M$ contains through every point a translate of $P$. Identifying $P$ with a coordinate subspace of $\M$, we have $M=M_0\times P$, where $M$ is embedded in $\M=P^\perp\times P$ as the product of a hypersurface embedding of $M_0$ in $P^\perp$ and the identity map of $P$. Thus $M$ is geodesically connected if and only if $M_0$ is geodesically connected. 
If $P^\perp$ is Euclidean, then $M_0$ is Riemannian and complete, hence geodesically connected.
Thus we need only consider timelike convex hypersurfaces $M$ of\, $\M$ that are not ruled by parallel timelike or spacelike lines.
\vspace*{3mm}

\noindent{\bf Case 1.}
\emph{Suppose $M$ is not ruled by parallel null lines.}

By Lemma \ref{lem:weak-strong}, there is a point $p\in M$ of strict convexity. 

Thus we may take $M$ as described in Theorem \ref{thm:wu2}. Consider the proper convex function $ f : M \to \R $ given by $ f = x^1 |M $, as in Theorem \ref{thm:support-convex}.

\setcounter{claim}{0}
\begin{claim}\label{critical-orientable}
Any two points of the critical set $C$ of $f$ are joined by a geodesic of $M$, and $C$ has an oriented neighborhood in $M$. 
\end{claim}

The critical points of $ f = x^1 |M $ are the points at which the $n$-plane tangent to $M$ has the form $x^1=c$. Since $M$ bounds a convex body $B$, it follows that $C=M\cap \{x^1=0\}=B\cap \{x^1=0\}$, and $C$ is a compact convex set. 
 
A sufficiently small neighborhood of $C$ in $M$ is diffeomorphic by projection $\Pi$ to a neighborhood of $C$ in $T_{\mathbf 0}M$, and hence is oriented. 

\begin{claim}
For some non-critical value $a$, there is $p\in M_a-L_a$ such that the map 
$\psi_{(p,a)}\,:\,S_p\,M\to L_a$ has nonzero degree, where 
$L_a =M \cap \{ x^1 = a \}$, $M_a=M \cap \{ x^1\le a \} $, $SM$ is the unit tangent bundle of $M$ with respect to some choice of Riemannian metric, and $\psi_{(p,a)}(\mathbf v)$ is the first point at which the geodesic of $M$ with initial velocity $\mathbf v\in S_pM$ leaves $M_a$. \end{claim}

Choose $p\in C$. Since $\Pi(M_a)$ is convex in $\{x^1=0\}$, each geodesic in $\{x^1=0\}$ from $p$ strikes $\Pi(L_a)$ transversely, and the corresponding map from $S_pM$ to $\Pi(L_a)$ has degree 1. In $\{x^1=0\}$ and in $M$ respectively, the vectors in $SM$ tangent to geodesics from $p$ agree on $C$. Since $C$ is compact and $M_a$ is arbitrarily close to $C$ for $a$ sufficiently small, it follows that each geodesic in $M$ from $p$ strikes $L_a$ transversely when $a$ is sufficiently small, and $\psi_{(p,a)}$ has degree $1$.

\begin{claim}\label{geodesic-leaves-sublevel} 
$M$ is disprisoning.
\end{claim}

Since $M$ is a topologically embedded timelike submanifold of $\M$ and $\M$ is strongly causal, 
$M$ is strongly causal, as follows for instance from Lemma \ref{lem:strong-causal-subman}. Therefore the claim holds for non-spacelike geodesics \cite[Proposition 3.13]{bee}.

Then $M$ satisfies the hypotheses of Lemma \ref{lem:exit-or-level}. 
Accordingly if the claim fails there is a complete non-constant geodesic $\gamma:\R\to M$ such that $f\circ\gamma $ is constant. $\gamma$ does not lie in the critical set $C$, since $C$ is compact and the geodesics of $C$ run on straight lines in $T_{\mathbf 0}M$. Therefore $\gamma$ lies in a level set $L_a$. Since $L_a$ is compact 
and $M$ is 
disprisoning on non-spacelike geodesics,
 it follows that $\gamma$ is spacelike.

Moreover, $\gamma$ cannot be a complete straight line. Let $J$ be the nonempty open subset of $\R$ consisting of all $t$ for which $\gamma''(t)\ne \mathbf 0$. For $t\in J$,
$$\gamma''(t)\,\in\, \{ x^1 = a \} \,\cap\,
(T_{\gamma(t)}M)^{\perp}.$$ Thus $T_{\gamma(t)}M$ is vertical, i.e. in the notation of Theorem \ref{thm:wu2}\,(\ref{graph}), $$(\Pi\circ\gamma)|\,J\ \subset \,D- \inte D.$$ By Theorem \ref{thm:wu2}\,(\ref{halfline}), $M\cap \Pi^{-1}(\gamma(t))$ contains the closed vertical half-line with endpoint $\gamma(t)$ and lying above $\gamma(t)$, for all $t$ in the closure of $J$, $\,\cl J$. 

Let $I$ be any maximal nonempty open subinterval of $\R-\cl J$. By compactness
of $L_a$,
$I$ is a finite interval. If $t$ is an endpoint of $I$, then $M\cap T_{\gamma(t)}M$ contains the line segment $\gamma(I)$ and the vertical half-line above $\gamma(t)$. Thus since $B$ is convex, the vertical planar strip above $\gamma(I)$ lies in $B\cap T_{\gamma(t)}M$, and in fact lies in $M\cap T_{\gamma(t)}M$ since $T_{\gamma(t)}M$ is a support hyperplane. It follows that 
$M$ contains a 
rolled
Euclidean half-plane,
a contradiction 
proving claim \ref{geodesic-leaves-sublevel}.

\vspace{2mm}

By Lemma \ref{lem:semi-geo-conn}, 
Case 1 
follows from these three claims.
\vspace*{3mm}

\noindent{\bf Case 2.}
\emph{Suppose $M$ is ruled by parallel null lines.}

\setcounter{claim}{0}
\begin{claim}\label{null-splitting}
 Let $\mathbf v \in S^{n-1}\subset\E^n$ and $C $ be a hypersurface of $\E^n$ 
having no tangent hyperplane orthogonal to $\mathbf v$. Define the hypersurface $L_{C,\mathbf v}$ of $\M$ by 
\begin{equation}\label{eq:L}
L_{C, \mathbf v} = \{ (p + t\mathbf v,t) \in \E^n \times (-\R) = \M : p \in C, t\in \R \}.
\end{equation}
Then $L_{C,\mathbf v}$ is a timelike hypersurface of $\M$ ruled by 
parallel
null lines. If $C$ is a convex hypersurface of $\E^n$, then $L_{C,\mathbf v}$ is a convex hypersurface of $\M$.

Furthermore, any timelike hypersurface $M$ of $\M$ that is ruled by parallel null lines can be constructed in this way, for some $\mathbf v \in S^{n-1}$ and some hypersurface $C \subseteq \E^n$ having no tangent plane orthogonal to $\mathbf v$.  $M$ is a convex hypersurface if and only if $C$ is a convex hypersurface.
\end{claim}

Denote $L_{C,\mathbf v}$ in (\ref{eq:L}) by $L$. Let $N$ be a local normal vector field of $C$ in $\E^n$. Then $\,N_L(p+t\mathbf v,t):\,= N(p) + \langle N(p), \mathbf v \rangle \frac{\partial}{\partial t}\,$ is a local normal field of $L$.
Since  $ \langle N,\mathbf v \rangle^2< \langle N,N \rangle \langle \mathbf v,\mathbf v \rangle =  \langle N,N \rangle$, then $\langle N_L, N_L \rangle > 0$.  Since the normals to $L$ are spacelike, then $L$ is timelike.

If $C$ is the boundary of a convex body $B$ in $\E^n \subseteq \M$, then $L$ is the boundary of the convex body $\{ (x,0) + s(\mathbf v,1) : x \in B, \,s \in \R \}$ in $\M$, where $(\mathbf v,1)$ is the null vector whose orthogonal projection into $\E^n$ is $\mathbf v$. Additionally 
$L$ is ruled by the parallel null lines 
$\{ (p,0) + s(\mathbf v,1): s \in \R \}$ for $p \in C$.

To see that any timelike hypersurface $M$ of $\M$ ruled by parallel null lines can be constructed in this way, choose $\mathbf v$ so that $\mathbf v + \frac{\partial}{\partial t}$ is a 
vector field tangent to the parallel null lines associated to $M$, and let $C = M \cap \E^n $ where $\E^n$ is the spacelike hyperplane $\{ x^{n+1} = 0 \}$. It is obvious that $C$ is a convex hypersurface of $\E^n$ if and only if $M$ is a convex hypersurface of $\M$.

\begin{claim}\label{convex-null-splitting}
Let $L_{C,\mathbf v}$ be the timelike hypersurface of $\M$ 
defined by (\ref{eq:L}),
where $\mathbf v\in S^{n-1}$ and $C$ is a 
closed embedded 
hypersurface of $\E^n$ that is nowhere orthogonal to $\mathbf v$. Then $L_{C,\mathbf v}$  is geodesically connected if and only if 
\begin{enumerate}
\item[$(\diamond)$] \label{prop:diamond}
any pair of points in $C$ can be connected by a $C^1$ curve $\alpha$ in $C$ such that $\langle \alpha', \mathbf v \rangle$ either has constant sign or vanishes identically.
\end{enumerate}
\end{claim}
(By {\em closed embedded} we  mean topologically embedded as a closed subset.)

Note that $K = (\mathbf v + \frac{\partial}{\partial t})|_{L_{C,\mathbf v}}$ is a null Killing field on $L_{C,v}$.
$K$ is complete since its integral curves are the null lines that rule $L_{C,\mathbf v}$. 

Additionally, $L_{C,\mathbf v}$ inherits global hyperbolicity from $\M$. To see this, let $S$ be the intersection of any Cauchy hypersurface of $\M$ with $L_{C,\mathbf v}$. 
Given a curve $\varphi$ in $L_{C,\mathbf v}$, we can write
$\varphi$, using a pair of curves $\alpha : I \to C$ and $\beta : I \to \R$, as $\varphi(s) = (\alpha(s) + \mathbf v \beta(s), \beta(s))$. 
Since any inextendible curve of $C$ is inextendible in $\E^n$, it follows 
that any inextendible causal curve  
of $L_{C,\mathbf v}$ is causal and inextendible in $\M$.
Thus $S$ is a Cauchy hypersurface for $L_{C,\mathbf v}$ and $L_{C,\mathbf v}$ is globally hyperbolic.

By Theorem \ref{thm:null-killing}, a pair of points in $L_{C,\mathbf v}$ are joined by a geodesic if and only if they are joined by a curve $\varphi$  such that the sign of $\langle \varphi', K \circ \varphi \rangle$ is constant or vanishes identically.
For a pair of points $(p+t_p\mathbf v,t_p), (q + t_q \mathbf v, t_q) \in L_{C,\mathbf v}$, 
let $\beta$ be any curve from $t_p$ to $t_q$.
The curve $\varphi(s) = (\alpha(s) + \mathbf v \beta(s), \beta(s))$ in $L_{C,\mathbf v}$ satisfies
$$
\langle \varphi', K \circ \varphi \rangle = 
\langle (\alpha' + (\mathbf v + \tfrac{\partial}{\partial t}) \beta' , \mathbf v + \tfrac{\partial}{\partial t} \rangle
= \langle \alpha', \mathbf v \rangle.
$$ 
Hence the lemma.

\begin{claim}\label{null-geo-conn}
Any timelike convex hypersurface $M$ of $\M$ 
that is ruled by parallel null lines
is geodesically connected.
\end{claim}

By Claim 1, we may write $M = L_{C,\mathbf v}$ as in (\ref{eq:L}), where $\mathbf  v\in S^{n-1}$ and $C$ is a 
convex
hypersurface of $\E^n$ which is nowhere orthogonal to $v$. By  
Claim 2, it suffices to verify  $(\diamond)$.

Define $f:C\to\R$ by $ f =  \langle \cdot, \mathbf v \rangle $. Since $C$ is nowhere orthogonal to $v$, then $f$ is regular.  Denote the level $f=a$ by $C_a$.  
Then for any $p\in C_a$, the gradient flow of $f$ provides neighborhoods $U_a$ in $C_a$ and  $U$ in $C$, and  $\varepsilon>0$, such that $U$ is diffeomorphic to $U_a\times (-\varepsilon,\varepsilon)$.  Call $U$ a \emph{splitting neighborhood}.

Suppose $p\in C_a$, $q\in C_b$.  If $a=b$, then $p$ and $q$ are joined by a curve $\alpha$  in $C_a$, so $\langle \alpha', \mathbf v \rangle=0$ and $(\diamond)$ is satisfied.
Since $C$ is a convex hypersurface of $\E^n$, the level sets of $f$ are convex hypersurfaces of  $(n-1)$-planes orthogonal to $ \mathbf v$, and hence are connected.  It follows that if 
$a\ne b$, and $p$ and $q$
are joined by a curve $\alpha$ as in $(\diamond)$,  then  $p$ and any $q'\in C_b$ are joined by such a curve.  To see this, join $q$ to $q'$ by a $C^1$ curve in $C_b$, covered by finitely many  splitting neighborhoods.  Then we obtain the desired curve by moving sufficiently close to $q$ along $\alpha$ and then moving through the splitting neighborhoods to $q'$.

Given $p\in C$, 
without loss of generality say $p\in C_0$.
Let $I$ be the maximal subinterval of $\range f$ 
containing $0$  
such that for all $c\in I$, $p$ is joined to some  (and hence every) point of $C_c$ by a curve satisfying ($\diamond$). The existence of splitting neighborhoods implies  that  $I$ contains an open interval about each of its members, and hence $I=(a,b)$ for $-\infty\le a<b\le\infty$;  and furthermore that $a, b\notin\range f$. Thus $I$ is both open and closed in $\range f$.

Since $C$ is a convex hypersurface of $\E^n$, $\range f$ is connected.  Thus $I=\range f$.  This completes the proof of Case 2, and hence of Theorem \ref{thm:main}.
\qed

\vspace{3mm}

Finally, let us verify a claim from Section \ref{introduction}:

\begin{prop}\label{prop:hyp-R>0}
The timelike convex hypersurfaces $M$ of $\M$satisfy $\mathcal{R}\ge 0$.
\end{prop}

\begin{proof}
If
$$ \II :T_p M \times T_p M\to (T_p M)^\perp$$ 
is 
the second fundamental form of $M$ in $ \M$, then the Gauss Equation states 
\begin{equation}
\label{eq:gausseqn}
R(\mathbf x,\mathbf y,\mathbf x,\mathbf y)=\langle \II (\mathbf x,\mathbf x), \II(\mathbf y,\mathbf y)\rangle - \langle \II(\mathbf x,\mathbf y),\II(\mathbf x,\mathbf y)\rangle. 
\end{equation}
Locally, a timelike convex hypersurface is the graph of a convex Lorentzian function $f : U \to \R$ where $U$ is a neighborhood of $\mathbf 0 \in T_p M \cong \m$. A simple calculation yields \begin{equation}
\label{eq:curveqn}
R(\mathbf x,\mathbf y,\mathbf x,\mathbf y)= \frac{\hess f(\mathbf x,\mathbf x) \hess f(\mathbf y, \mathbf y) - \hess f(\mathbf x,\mathbf y)^2}{1 + \langle (\del f)_p, (\del f)_p \rangle}
\end{equation} The numerator is nonnegative by convexity of $f$ and the denominator is positive
 because $f$ is Lorentzian. It follows that timelike sectional curvatures are $\le 0$ and spacelike sectional curvatures are $\ge 0$.
\end{proof}

%%%%%%%%%%%%%%%%%%%%%%%%%%%%%%%%%%%%%%%%%%%%%%%%%

\section{Existence criterion for convex functions}
\label{sec:convex}

The aim of this section is to prove a criterion for the non-critical level sets of a (proper) not-necessarily-convex function $u : M \to \R$ on a semi-Riemannian manifold to be the levels of a (proper) convex function $f : M \to \R$. We use this criterion to extend
our theorems on geodesic connectedness. Finally we apply these results to a class of not-necessarily-convex hypersurfaces in Minkowski space.

The corresponding Riemannian criterion is given in \cite{ab-convex}.

\begin{thm}\label{thm:convex-function1}
Suppose $u:M\to\R$ is a smooth function on a semi-Riemannian manifold $M$. Denote the critical set of $u$ by $\,C$. 
For every $\,a\in \range u\,$, we denote the non-critical points of $\,u^{-1}(a)$ by $\,L_a$. 

Suppose $N$ is a vector field on $M - C$ satisfying $Nu>0$. Define $\eta:M-C\to\R$ 
by $$ \eta(p)=\hess u(N_p,N_p)/ (N_pu)^2. $$
Define 
$\mu:\,\{a\in\R:L_a\ne\emptyset\}\to\R$\, by 
\begin{align}
\label{eq:convex-test}
\mu\,(a)\, =\, \inf \: \bigl\{ \eta(p), \eta(p)- \bigl(\hess\, u\,(\mathbf x,N_p)^2/& (N_pu)^2\hess u\,(\mathbf x,\mathbf x)\bigr) :\\&
p \in L_a, \mathbf x\in T_pL_a,\,\hess u\,(\mathbf x,\mathbf x)\ne 0 \bigr\}.
\notag
\end{align}

\vspace{1mm}

Then there is a smooth function $\rho:\range u\to\R$ such that 
$\rho' \geq 1$ 
and $f=\rho \circ u$ is convex if and only if $u$ satisfies the following conditions (\ref{criticals})--(\ref{bound-below}):
\begin{enumerate}
\item\label{criticals}
the critical set $C$ consists of local minimum points,

\item\label{infin-convex}
the restriction of $\,\hess u$ to each tangent space $T_pL_a$ is positive semidefinite (i.e. each $L_a$ is a locally convex hypersurface on which $N$ is outward-pointing), 

\item\label{null}
if $\mathbf x\in T_pL_a$ and $\hess u\,(\mathbf x,\mathbf x)=0$, then $\hess u\,(\mathbf x,N_p)=0$,

\item\label{bound-below}
the function $\mu$ is bounded below by a
function that extends to a 
continuous function \,$h:\range u
\to\R$ (which clearly may be assumed nonpositive). \end{enumerate}

Moreover, $f=\rho \circ u$ is strictly convex if and only if $u$ satisfies 
condition (\ref{bound-below}) and the following conditions (\ref{hyp:criticals2}) and (\ref{infin-strictly-convex}):
\begin{enumerate}[(i)]
\item\label{hyp:criticals2}
the critical points are nondegenerate local minima,

\item\label{infin-strictly-convex}
the restriction of $\,\hess u$ to each tangent space $T_pL_a$ is positive definite.
\end{enumerate}

Finally, if $u$ is proper, then $f$ may be assumed proper. 
\end{thm}

\begin{remark}

Since $\,\hess u$ is positive semidefinite at each $p \in 
L_a$, condition (\ref{null}) means the nullspace of $\,\hess u$ on $T_pL_a$ lies in the nullspace of $\,\hess u$.
\end{remark}
 
\begin{remark}
As mentioned in the introduction, it is often possible to verify (\ref{bound-below}) by showing that \,$\mu$\, is continuous and finite. \end{remark}

\begin{proof} In \cite[Theorem 1]{ab-convex}, Alexander and Bishop proved this result when $M$ is a Riemannian manifold and $N = \del u / \| \del u \|$, the outward unit normal field to the level sets of $u$. However, it suffices to require only $Nu > 0$ and $M$ semi-Riemannian. Then the calculations in the proof are unchanged.

In particular, the convexity condition of $f$ is shown to imply that we may take $\rho$ to be a solution of the differential equation $ \rho'' + h \rho' = 0 $ on $ \range u $. Since $h$ is nonpositive, $\rho$ can be chosen so that $ \rho'(a) = \exp({ -\int_{a_0}^a h}) \geq 1 $ where $a_0 = \inf \range u$. Therefore $\rho$ is proper. Thus $f$ will be proper if $u$ is proper.
\end{proof}

Now we have the following analogue of
%S cross-referencing needs to be changed in proofs 
Theorem \ref{thm:gordon-semi} and Corollary \ref{cor:lorentz-graph}:

\begin{thm}\label{thm:warbled-geo-connected}
Let $M$ be a null-disprisoning semi-Riemannian manifold. Suppose $M$ supports a proper nonnegative function $u : M \to \R$ whose critical set is a point and for which there is no non-constant complete geodesic on which $u$ is constant. 

Suppose $u$ satisfies conditions (1),(2),(3),(4) of Theorem \ref{thm:convex-function1} for a vector field $N$ on $M - C$ satisfying $Nu > 0$. Then $M$ is geodesically connected.

If in addition $M$ is a strongly causal space-time and $u$ is a Lorentzian function, then the graph $\Gamma(u)$ is geodesically connected.
\end{thm}
 
\begin{proof} 
\setcounter{claim}{0}
Since $ u : M \to \R $ satisfies the conditions of Theorem \ref{thm:convex-function1} for the vector field $N$, we can choose a smooth function $ \rho : \range u \to \R $ such that the function $ f_M = \rho \circ u $ is proper, convex, and has the same critical set, level sets, and sublevel sets as $ u $. Since the critical set of $f_M$ is a point, and there is no non-constant complete geodesic on which $f_M$ is constant, it follows from Theorem \ref{thm:gordon-semi} that $M$ is geodesically connected.

Now assume further that $M$ is a strongly causal space-time and $u$ is a Lorentzian function. By Lemma \ref{lem:strong-causal-subman}, $\Gamma(u)$ is strongly causal. Thus $\Gamma(u)$ is null-disprisoning \cite[Proposition 3.13]{bee}.

Since projection $ \Pi: \Gamma(u)\to M$ by $(p, u(p)) \mapsto p $ is a diffeomorphism, we may identify a vector $\mathbf x\in T_pM$ with its lift, $\mathbf x + (\mathbf x u) \partial_y$, via the inverse projection map. Here we write $\mathbf x$ for either. We write $ v : \Gamma(u) \to \R $ for the lift of $u$ to $\Gamma(u)$. Thus $ Nu = Nv > 0$. The 
non-critical 
level sets of $u$ and their lifts to level sets of $v$ will both be denoted by $ L_a $, and similarly for the sublevel sets $M_a$ and the critical set $C$. Thus $C$ denotes either the minimum point of $u$ on $M$ or the minimum point of $v$ on $\Gamma(u)$.

\begin{claim}
\label{clm:nonconvex-geo-conn}
The lift $ v : \Gamma(u) \to \R $ of $u$ to $\Gamma(u)$ satisfies the conditions of Theorem \ref{thm:convex-function1}. 
for the vector field $N$ on $\Gamma(u) - C$. Thus there is a smooth function $ \rho : \range v \to \R $ such that $ \rho' \geq 1 $ and $ f = \rho \circ v : \Gamma(u) \to \R $ is proper and convex with the same level sets and critical sets as $v$.
\end{claim}

Since $u$ is Lorentzian, $1 + \langle \del u, \del u \rangle > 0$ and by Lemma \ref{lem:convex-lift}, $$
\hess v(\mathbf x,\mathbf y) = \frac{\hess u(\mathbf x,\mathbf y)}{1 + \langle (\del u)_p, (\del u)_p \rangle}.
$$ 
Thus $ \hess v(\mathbf x,\mathbf x) \geq 0 $ if and only if $ \hess u(\mathbf x,\mathbf x) \geq 0 $. Since $L_a$ is infinitesimally convex in $M$ for any $a$, $L_a$ is also infinitesimally convex in $\Gamma(u)$.

Furthermore, $\hess v(\mathbf x,\mathbf y) = 0$ if and only if $ \hess u(\mathbf x,\mathbf y) = 0 $. 
If $\hess v(\mathbf x,\mathbf x) = 0 $, then $ \hess u(\mathbf x,\mathbf x) = 0 $ and $ \hess u(\mathbf x,N_p) = 0 $ by condition (3), and hence $ \hess v(\mathbf x,N_p) = 0 $.

Define $ \eta^u : M - C \to \R $ and $ \mu^u : \range( u \, | \, M - C) \to \R $ as $\eta$ and $\mu$ are defined for $u$ in Theorem \ref{thm:convex-function1}, and $ \eta^v : \Gamma(u) - C \to \R $ and $ \mu^v : \range( v \, | \, \Gamma(u) - C) \to \R $ similarly for $v$. Let $ h^u : \range u \to \R $ be a continuous lower bound of $ \mu^u $. Then by condition (4), 
$$
\eta^v(p) = \hess v(N_p, N_p)/(N_p v)^2 = \frac{\hess u(N_p, N_p)/(N_p u)^2}{1 + \langle (\del u)_p, (\del u)_p \rangle} \geq h^u(a)/w(a)
$$ where $p\in M_a$ and $ w(a) = \sup \{ 1 + \langle (\del u)_q, (\del u)_q \rangle : q \in M_a \} $. Similarly, if $ \hess v(\mathbf x,\mathbf x) \neq 0 $, then 
$$
\eta^v(p) - \frac{\hess v(\mathbf x, N_p)^2}{(N_p v)^2 \hess v(\mathbf x,\mathbf x)} \geq h^u(a)/w(a).
$$ 
Since $ u $ is Lorentzian, $ w(a) $ is a positive nondecreasing continuous function. Thus $ h^v(a) := h^u(a)/w(a) $ is a continuous lower bound of $ \mu^v(a) $ on $ \range v $, and claim \ref{clm:nonconvex-geo-conn} is verified.

\vspace{2mm}

Suppose $\gamma$ is a non-constant complete geodesic of $\Gamma(u)$ on which $f$ is constant. Since $\gamma$ is a curve in the graph of $ u $, we can find $ \alpha : \R \to M $ so that $ \gamma = (\alpha, u \circ \alpha ) $ is the lift of $\alpha$ to $\Gamma(u)$ and as in (\ref{geo-in-graph}), $$ \alpha''(t) + (u \circ \alpha)''(t)( \del u)_{\alpha(t)} = \mathbf 0. $$ 
However, since $ \gamma $ is in a level set of $f$ and therefore of $v$, $\alpha$ must be in a level set of $u$. Thus $ (u \circ \alpha)'' = 0 $, hence $ \alpha'' = \mathbf 0 $. Thus $\alpha$ is a non-constant complete geodesic of $M$ on which $u$ is constant, a contradiction.
Therefore by Theorem  
\ref{thm:gordon-semi}, $\,\Gamma(u)$ is geodesically connected.
\end{proof}

Finally we construct a large class of not-necessarily-convex but geodesically connected Lorentzian hypersurfaces. Specifically, we may perturb the levels of a strictly convex Lorentzian function by any $\sigma$ satisfying $\, 0 < \sigma' \le 1$ and still retain a geodesically connected graph: 

\begin{cor}\label{warbled-examples}
Let $ \sigma : [0, \infty) \to [0, \infty) $ be a proper smooth function with $0 < \sigma\,^{\prime}\leq 1 $ and let $ f : \m \to \R$ be a proper nonnegative strictly convex Lorentzian function. Let $ M \subset \M $ be the graph of $ u = \sigma \circ f : \m \to \R $ in $ \m \times \R = \M $. Then $M$ is a timelike hypersurface of $\M$ and is geodesically connected.
\end{cor}

\begin{proof}
$u$ is Lorentzian since 
$$ \langle \del u, \del u \rangle = \langle \del f, \del f \rangle (\sigma' \circ f)^2 > (-1) (\sigma' \circ f)^2 \geq -1.
$$

Let $ \rho = \sigma^{-1}$ and $ N = \del_{\, \text{Riem}} u $, the gradient of $u$ in any Riemannian metric $ g_{\text{\,Riem}} $ on $M$. Since $ f = \rho \circ u $ is strictly convex, there is no non-constant complete geodesic on which $u$ is constant. Since $\rho\,^\prime = 1/( \sigma\,^\prime \circ \sigma^{-1} ) \geq 1 $ and $ Nu > 0 $ on the non-critical 
set, then by Theorem \ref{thm:convex-function1}, $u$ satisfies $(1)$, $(2)$, $(3)$ and $(4)$ for $N$. 
By Theorem \ref{thm:warbled-geo-connected}, since $\m$ is null-disprisoning, $M$ is geodesically connected.
\end{proof}

\appendix
\section{Orthogonal splittings}
\label{app:conclusion}
Let us recall 
%S
the following theorem of Benci, Fortunato and  Masiello \cite[Theorem 1.1]{bfm}, which is
the comprehensive theorem on geodesic connectedness 
%S
in the orthogonal splitting case, as 
discussed in the survey by Candela and Sanchez \cite[Theorem 4.37]{cs:survey}:

\begin{definition}\label{def:OrthogonalSplitting}
A Lorentzian manifold $M$ is an \emph{orthogonal splitting space-time} if $M$ is isometric to $( M_0 \times \R, \langle \cdot, \cdot \rangle_L)$ where \begin{equation}
\langle \mathbf z, \mathbf z' \rangle_L = 
\langle A(z) \mathbf x, \mathbf x' \rangle_R - \beta(z) \mathbf t \mathbf t'
\end{equation} for any $ z = (x, \tau) \in M$ $(x \in M_0, \tau \in \R) $, and $ \mathbf z = (\mathbf x, \mathbf t), \mathbf z' = (\mathbf x', \mathbf t') \in T_z M \cong T_x M_0 \times \R$. Here $ (M_0, \langle \cdot, \cdot \rangle_R ) $ is a finite-dimensional, connected Riemannian manifold, $ A(z):T_x M_0\to T_x M_0 $ is a smooth, symmetric, strictly positive linear operator, and $ \beta : M \to \R $ is a smooth, strictly positive scalar field.
\end{definition}

\begin{thm} \cite{bfm}\label{thm:cs-main}
Let $M$ be an orthogonal splitting space-time, isometric to $( M_0 \times \R, \langle \cdot, \cdot \rangle_L)$, where $ (M_0, \langle \cdot, \cdot \rangle_R ) $ is a complete Riemannian manifold. Assume that there exist constants $a$, $b$, $c$, $d > 0 $ such that the coefficients $A$, $\beta$ in (\ref{def:OrthogonalSplitting}) satisfy the following hypotheses: \begin{gather}
a \langle \mathbf x, \mathbf x \rangle_R \leq \langle A(z) \mathbf x, \mathbf x \rangle_R, \label{hyp:Abound}\\
b \leq \beta(z) \leq c 
, \label{hyp:betabound} \\
| \beta_\tau(z)| \leq d, \quad| \langle A_\tau(z) \mathbf x, \mathbf x \rangle_R | \leq d \langle \mathbf x, \mathbf x \rangle_R, \label{hyp:atbtbound}
\end{gather} for all $ z = (x, \tau) \in M $, $ \mathbf x \in T_x M_0 $. Furthermore, assume that \begin{gather}
\limsup_{\tau \to +\infty} ( \sup \{ \langle A_\tau(z) \mathbf x, \mathbf x \rangle_R : x \in M_0, \, \mathbf x \in T_x M_0, \langle \mathbf x, \mathbf x \rangle_R = 1 \} ) \leq 0, \\
\liminf_{\tau \to -\infty} ( \inf \{ \langle A_\tau(z) \mathbf x, \mathbf x \rangle_R : x \in M_0, \, \mathbf x \in T_x M_0, \langle \mathbf x, \mathbf x \rangle_R = 1 \} ) \geq 0.
\end{gather} Then $M$ is geodesically connected.
\end{thm}

%S
\begin{remark}
Under an assumption of exhaustion of $\,M\,$ by convex strips, Masiello weakens the assumptions of Theorem \ref{thm:cs-main} on the coefficients of the orthogonal splitting that guarantee geodesic connectedness \cite{masiello:convex}.
%S ???
It  is not apparent
that  general timelike convex hypersurfaces carry  such exhaustions.
\end{remark}

\setcounter{figure}{0} 
In this section, we examine the conditions of Theorem \ref{thm:cs-main} in 
two natural examples of splittings of a strictly convex hypersurface $M$. We take $M$ to be the graph in $\Mink$ given by 
$$x^2 =f(x^1,t)= \sqrt{(x^1)^2 + t^2 + 1}.$$ 
By Theorem \ref{thm:main}, $M$ is geodesically connected. We show that the splittings do not satisfy the conditions of Theorem \ref{thm:cs-main}, nor do we know of any other splittings that do so.

\begin{figure}[h!]
\centering
\includegraphics[width=0.6\textwidth]{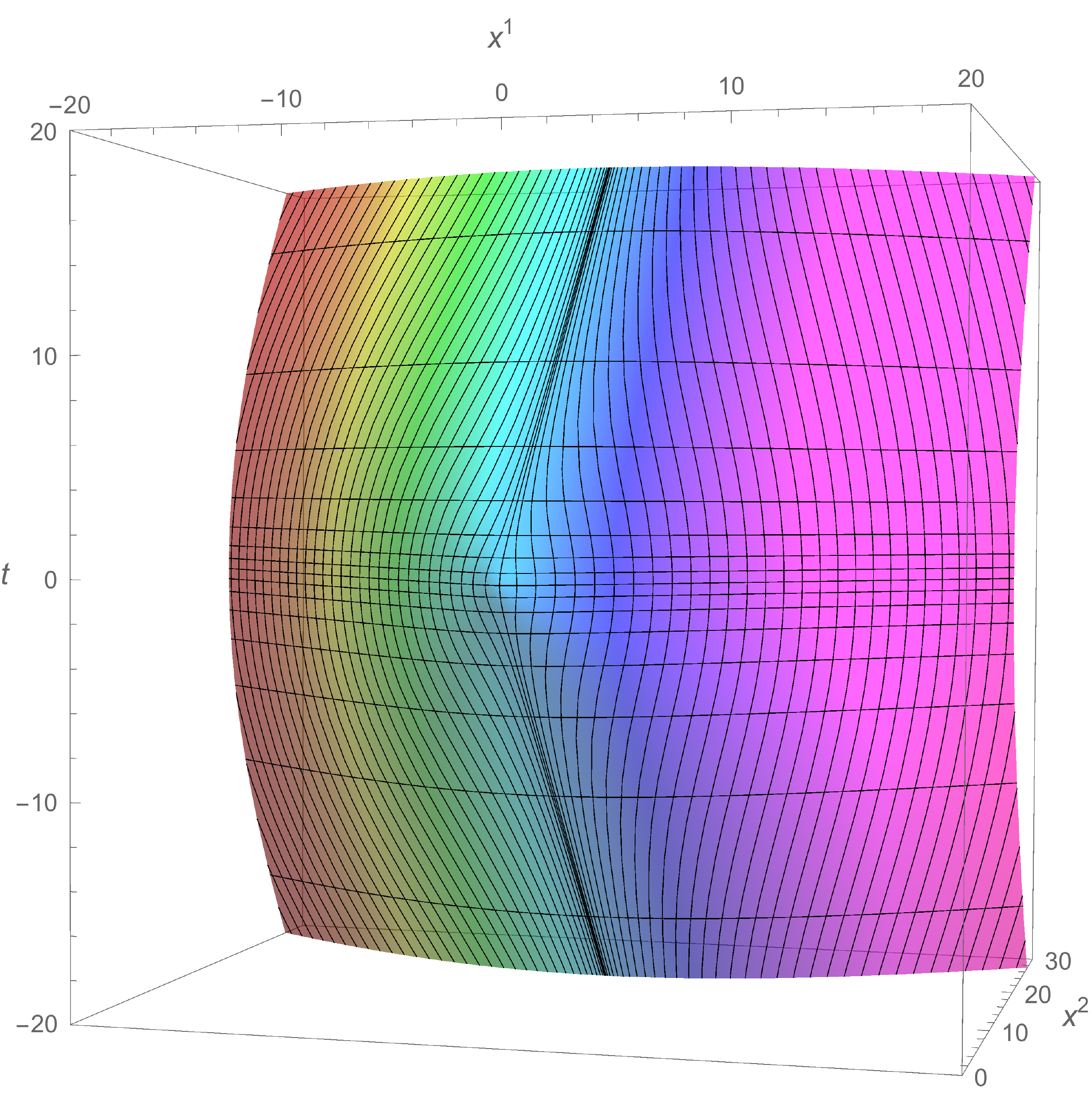}
\caption{
The timelike convex hypersurface $M$ 
with orthogonal splitting coordinates generated by taking $\tau((x^1,f(x^1,t),t)) = \sinh^{-1}(t)$ as time function.}
\label{fig:orth1}
\end{figure}

Figure \ref{fig:orth1} illustrates the first example. Intersecting $M$ with hyperplanes $t=$ constant gives a family of Cauchy hypersurfaces which we parametrize with a time function $\tau$ as follows:
$$M_\tau = M \cap \{ (x^1,x^2,t) \in \Mink : t = \sinh(\tau) \}.$$

The orthogonal trajectories $ \gamma_x(t) $ are illustrated, obtained by solving a family of ordinary differential equations numerically. In order for the coordinates $$(x, \tau) \mapsto (x^1(x,\tau), f(x^1(x,\tau), \sinh(\tau)), \sinh(\tau)) $$ to form an orthogonal splitting, $x^1$ must satisfiy the differential equation $$ x^1_\tau(x, \tau) = - \frac{x^1(x,\tau) \sinh(\tau)\cosh(\tau) }{\cosh^2(\tau) + 2 (x^1(x,\tau))^2 }.$$ 
Then
 $$\beta(x, \tau) = \frac{(1+ 2(x^1(x,\tau))^2) \cosh^2(\tau)}{2(x^1(x,\tau))^2 + \cosh^2(\tau)}.$$
Since $\beta$ is unbounded along the curve $x^1 = \cosh(\tau)$, hypothesis (\ref{hyp:betabound}) of Theorem \ref{thm:cs-main} is violated. 
Additionally, one can show that $A(x, \tau) \to 0$ as $\tau \to \pm \infty$ along the meridian $x^1 = 0$, so the condition $a >0$ on $A$ in hypothesis (\ref{hyp:Abound}) is violated. 
The splitting can be modified by reparameterizing the time function $\tau$ but this does not effect the boundedness of $A$.

\begin{figure}[h!]
\centering
\includegraphics[width=0.6\textwidth]{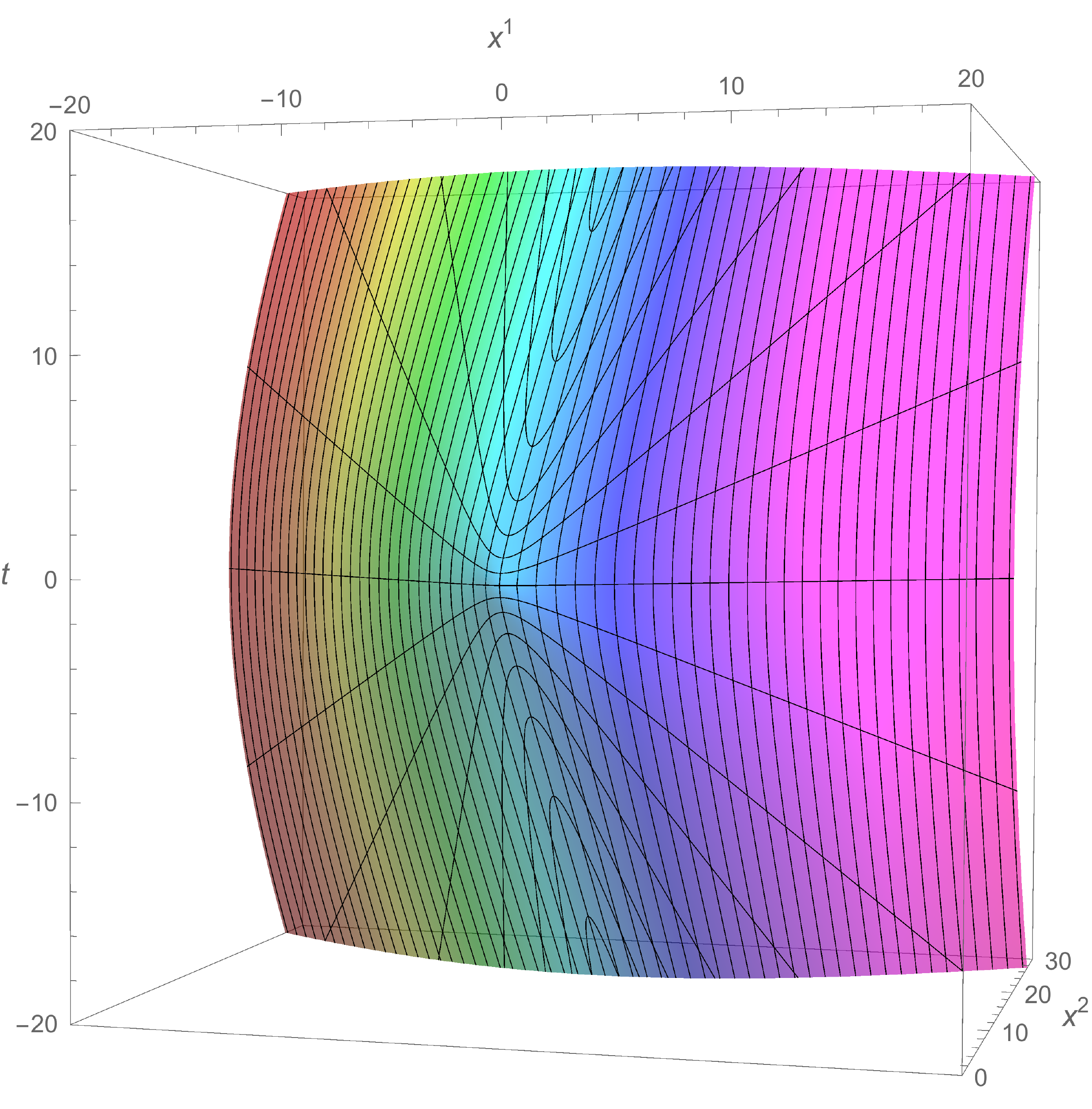}
\caption{
The timelike convex hypersurface $M$ 
with orthogonal splitting coordinates given by boost coordinates.}
\label{fig:orth2}
\end{figure}

Another natural splitting, illustrated in Figure \ref{fig:orth2},
is obtained by boosting the $t = 0$ slice about the $x^1$-axis, namely 
$$ (x, \tau) \mapsto (x, \cosh(\tau)\sqrt{1 + x^2}, \sinh(\tau)\sqrt{1 + x^2} ).$$ We obtain $ A(x, \tau)\mathbf x = ((1 + 2x^2)/(1 + x^2)) \mathbf x$ and $\beta(x, \tau) = 1 + x^2$. Thus $\beta$ is unbounded above as $x \to \pm \infty$.
This splitting can be modified by reparameterizing the time function $\tau$ but $\beta$ will remain unbounded as $x \to \pm \infty$.
With this particular splitting
 the metric is static, i.e. $A$ and $\beta$ are independent of $\tau$, and other methods discussed in \cite{cs:survey} can be applied to show geodesic connectedness.

However, there may well exist non-static hypersurfaces satisfying the conditions of our main theorem, on which an orthogonal splitting metric cannot 
simultaneously 
satisfy the required bounds on $A$ and $\beta$. In general it is unclear how to know if a given timelike hypersurface is 
static or 
has orthogonal splittings satisfying the conditions of Theorem \ref{thm:cs-main}.

%S 
\section*{Acknowledgments} We thank Gregory Galloway and Antonio Masiello for mentioning to us some previous works on convexity.

\newpage

%%%%%%%%%%%%%%%%%%%%%%%%%%%%%%%%%%%%%%%%%%%%%%%%%%%%%%%%%%%%%%%%%%%%%%%%%%%%%%%%%%%%%%%%%%%%%%

\end{document}